\documentclass{article}

\usepackage{hyperref}
\usepackage{amsmath}
\usepackage{amsthm}
\usepackage{amssymb}
\usepackage{url}
\usepackage{xcolor}
\usepackage{algorithm}
\usepackage{algpseudocode}
\usepackage[title]{appendix}
\usepackage[all]{xy}
\usepackage{geometry}
\usepackage{cancel}
\usepackage{comment}

\usepackage{array}

\newtheorem{theorem}{Theorem}[section]

\newtheorem{lemma}[theorem]{Lemma}
\newtheorem{proposition}[theorem]{Proposition}
\newtheorem{conjecture}[theorem]{Conjecture}
\newtheorem*{conjecture*}{\textbf{Conjecture}}
\newtheorem{corollary}[theorem]{Corollary}
\newtheorem{corollary*}[theorem]{Corollary 4.6 and Proof of Theorem 1.2}

\newtheorem{claim}{Claim}[theorem]
\newtheorem{fact}{Fact}[theorem]
\newcommand{\Fix}[1]{\text{Fix} (#1)}
\newcommand{\F}{\mathbb{F}}

\DeclareMathOperator{\Fr}{Fr}
\DeclareMathOperator{\Gl}{Gl}

\DeclareMathOperator{\Span}{span}

\DeclareMathOperator{\lcm}{lcm}

\newcommand{\Nir}[1]{\textcolor{red}{Nir: #1}}

\theoremstyle{definition}
\newtheorem{definition}{Definition} 
\newtheorem{example}{Example}[theorem]
\newtheorem*{notation}{\textbf{Notation}}
\newtheorem{remark}{Remark}[theorem]

\providecommand{\keywords}[2]
{
  \small	
  \textit{Keywords: } #1.
  
  \textit{2020 AMS Subject Classification:} #2.
}

\author{Adithya Balachandran, Nir Gadish\footnote{N.G. is supported by NSF Grant No. DMS-1902762.} , Andrew Huang, Siwen Sun}
\date{May 25, 2021}

\title{Product Expansions of $q$-Character Polynomials}
\begin{document}

\maketitle

\begin{abstract}
The ring of $q$-character polynomials is a $q$-analog of the classical ring of character polynomials for the symmetric groups. This ring consists of certain class functions defined simultaneously on the groups $\Gl_n(\F_q)$ for all $n$, which we also interpret as statistics on matrices.
Here we evaluate these statistics on all matrices and work towards computing the structure constants of the product in this ring.

We show that the statistics are periodically polynomial in $q$, and governed by universal polynomials $P_{\lambda,\mu}(q)$ which we compute explicitly, indexed by pairs of integer partitions. The product structure is similarly polynomial in $q$ in many cases, governed by polynomials $R_{\lambda,\mu}^{\nu}(q)$ indexed by triples of partitions, which we compute in some cases.

Our calculations seem to exhibit several unexpected patterns. Mainly, we conjecture that certain indecomposable statistics generate the whole ring, and indeed prove this for statistics associated with matrices consisting of up to 2 Jordan blocks. Furthermore, the coefficients we compute exhibit surprising stability phenomena, which in turn reflect stabilizations of joint moments as well as multiplicities in the irreducible decomposition of tensor products of representations of $\Gl_n(\F_q)$ for $n\gg 1$.
We use this stabilization to compute the correlation of the number of unipotent Jordan blocks of two sizes.

\end{abstract}

{20G40, 05E05, 05A10}

\section{Introduction}
In this paper we study the ring of $q$-character polynomials -- a collection of class functions on invertible matrices $\Gl_n(\F_q)$, where $\F_q$ is a field of $q$ elements and $n$ ranges over all positive integers. These are $q$-analogs of character polynomials of symmetric groups, introduced in \cite{1803.04155}, and they arise as the stable characters of finitely generated sequences of representations $\Gl_n(\F_q)\curvearrowright V_n$ as $n\to \infty$ known as $\mathtt{VI}$-modules.
When considered as statistics on matrix groups, these functions exhibit a uniform behavior as $q$ and $n$ vary. The main goal of this paper is to work towards understanding the product structure of this ring relative to a particular natural basis.

As a motivating example, consider the following invariant of matrices:
$$\Fix{B}_q=|\{v \in \F_q^n \mid Bv=v \}|.$$
Fulman and Stanton \cite{fulman-fixed} compute the distribution of this function as a statistic on $\Gl_n(\F_q)$ as a uniform probability space. One notable consequence of their calculation is that the moments of this statistic become independent of $n$ as $n\to\infty$ (e.g. the expectation is always $2$). This pattern is explained in \cite{1803.04155} as a general stability phenomenon of $q$-character polynomials, of which $\Fix{\cdot}_q$ is the most basic example. The general case has fixed vectors replaced by fixed subspaces -- let us recall the general definition.

\begin{definition}[\textbf{$q$-Character Polynomials}, \cite{1803.04155}]
Given a finite field of $q$ elements $\F_q$, let $A$ be any fixed $m\times m$ matrix. The \emph{$q$-character polynomial} associated to $A$ is the following statistic, defined simultaneously on square matrices of any size:
\[
\text{for }B\in \operatorname{End}(\F_q^n) \quad  X_A(B,q) = |\{ W \leq \F_q^n | \dim W = d \text{ satisfies } B(W)\subseteq W \text{ and } B|_W \sim A \} |.
\]
Here $B|_W\sim A$ refers to matrix similarity. Furthermore, since $A$ and $B$ are also transformations over field extensions $\F_{q^m}$, we similarly consider $X_A(B,q^m)$ for all $m\geq 1$. When the cardinality of the field is clear from context, we will suppress it from the notation and write $X_A(B)$.

A \emph{$q$-character polynomial} is any $\mathbb{Q}$-linear combination of functions $X_A : \bigcup_{n\geq 1} \Gl_n(\F_q) \to \mathbb{Q}$.
\end{definition}
One of the main contributions of this paper is the complete evaluation of $q$-character polynomials. Interestingly, though the functions are not defined as such, they are often polynomial functions in $q$.

\begin{theorem}[\textbf{Periodic Polynomiality}]\label{theorem-main:polynomiality}
If $A$ and $B$ are matrices over $\F_q$, the statistics $X_A(B,q^m)$ coincide with integer polynomials in $q^m$ in a periodic manner. Explicitly, if $d$ is the degree of the splitting field of the characteristic polynomial of $A$ over $\F_q$, then there exist polynomials $P_1(t),\ldots,P_d(t)\in \mathbb{Z}[t]$ such that if $m\equiv i$ modulo $d$ then
\begin{equation}
    X_A(B,q^m) = P_i(q^m).
\end{equation}
In particular, the zeta function generated by these statistics is rational, though we omit the proof.

Furthermore, if $\lambda$ and $\mu$ are two integer partitions, let $J_\lambda$ and $J_\mu$ be unipotent matrices in Jordan form with block sizes specified by the respective partition (and so defined over all fields). Then there exists a single polynomial $P_{\lambda,\mu}(t)$ which coincides with the statistic $X_{J_\lambda}(J_\mu,p^m)$ for all prime powers.
\end{theorem}
In fact, we show that the polynomials $P_{\lambda,\mu}(t)$ associated with statistics of unipotent matrices determine all other statistics $P_{A,B}(t)$ -- see Theorem \ref{theorem: reducing to unipotents} for the precise formula and a proof of Theorem \ref{theorem-main:polynomiality}.  Some readers might be interested to learn that the periodic list of polynomials in Theorem \ref{theorem-main:polynomiality} takes a form familiar from number theory and combinatorics: in many cases $P_i(t) = P_{\lambda,\mu}(t^{\lcm(i,d)})^{\gcd(i,d)}$ for some fixed $d\geq 1$.
In \S\ref{sec: evaluations} we compute the polynomials $P_{\lambda,\mu}(t)$ explicitly, hence all evaluations $X_A(B)$ are completely determined.

Moving on to products, it is shown in \cite[Corollary 3.9]{GADISH2017450} that the collection of $q$-character polynomials is closed under pointwise multiplication. Combined with the stabilization of expectation \cite[Corollary 4.6]{GADISH2017450}, this product structure is responsible to the aforementioned stabilization of higher joint moments of $X_A$ and similar $q$-character polynomials. Recalling that $q$-character polynomials are characters of certain sequences of representations $\Gl_n(\F_q)\curvearrowright V_n$, the expectation of products also describe multiplicities in the irreducible decomposition of $V_n$, which therefore similarly stabilize as $n\to \infty$ (see \cite{GADISH2017450} for details). 

To the best of our understanding, the product structure of $q$-character polynomials has not been previously explored, and is the original motivation for this work. Our main results regarding this structure take a similar form to Theorem \ref{theorem-main:polynomiality}.
\begin{theorem}[\textbf{Product expansions}] \label{theorem-main: product expansions}
There exist integer polynomials labeled by triples of partitions $R_{\lambda,\mu}^{\nu}(t)\in \mathbb{Z}[t]$, such that if $J_\lambda$ and $J_\mu$ are unipotent matrices in Jordan form with block sizes specified by the respective partitions, then
\[
X_{J_\lambda}\cdot X_{J_\mu} = \sum_{\nu} R_{\lambda,\mu}^{\nu}(q^d) X_{J_{\nu}}
\]
when considering the statistics as defined on matrices over the field $\F_{q^d}$.

For general matrices $A$ and $B$ the analogous expansion $X_A\cdot X_B$ is expressible in term of $R_{\lambda,\mu}^{\nu}$ and the characteristic polynomials of $A$ and $B$, so long as these characteristic polynomials split over $\F_q$. See Theorem~\ref{theorem: general coefficients in terms of unipotents} for a general formula.

\end{theorem}
The main tool that sheds light on the coefficients of generalized product expansions is a recursive algorithm we give in \S\ref{sec: recursive formula} for computing them. This approach allows us to posit a general criterion for constraining the conjugacy classes that may appear in the product expansions, see Lemma~\ref{lemma: meta}.
A proof of Theorem \ref{theorem-main: product expansions} is found immediately below Corollary \ref{corollary: polynomial expansion coeff}.

\begin{theorem}[\textbf{Recursive description}]\label{theorem-main:recursive description}
For every pair of integer partitions $\lambda$ and $\mu$, the polynomials $R_{\lambda,\mu}^\nu(t)$ which govern the product expansions of $q$-character polynomials satisfy the following recursion
\[
R_{\lambda,\mu}^{\nu}(t) = P_{\lambda,\nu}(t)\cdot P_{\mu,\nu}(t) - \sum_{\|\nu'\| < \|\nu\|} R_{\lambda,\mu}^{\nu'}(t) P_{\nu',\nu}(t).
\]
where the sum goes over all partitions $\nu'$ of numbers $<\|\nu\|$.
\end{theorem}
We prove this in \S\ref{sec: recursive formula}.

Our product expansions generalize classical problems in $q$-combinatorics. One explicit example with a natural interpretation as a product of $q$-character polynomials is the following.
\begin{example}[\textbf{Identity matrices}]
For every integer $n\geq 1$ let $I_n$ denote the $n\times n$-identity matrix. Then for every $n$ and $k\geq 1$, the evaluation $X_{I_k}(I_n)$ is simply counting the number of $n$-dimensional subspaces in $\F_q^m$, commonly denoted by the Gaussian binomial coefficient $\binom{n}{k}_q$. Hence expanding the product of two such statistics generalizes the product expansion of these $q$-binomial coefficients as linear combinations of other $q$-binomial coefficients:
\begin{align*}
X_{I_m}\cdot X_{ I_n}&=\sum_{k=0}^{\min{(m,n)}} \binom{m+n-k}{k}_q\binom{m+n-2k}{m-k}_q \cdot q^{(m-k)(n-k)}X_{I_{m+n-k}}.
\end{align*}
Evaluating this equation on a large identity matrix $I_N$ gives a uniform expansion for $\binom{N}{n}_q \binom{N}{m}_q$.

This equality is easily verifiable by a simple counting argument: the left-hand side is counting the number of ways to pick two subspaces $V$ and $W$ in $\F_q^{N}$ of dimensions $m$ and $n$ respectively. The right-hand side counts the same thing, but by first fixing the spaces $V+W$ and $V\cap W$, then enumerating the pairs $V$ and $W$ that could result in those.
\end{example}

In \S\ref{sec: expansion coefficients} we apply our recursive description for the product expansion coefficients to determine general formulas for the product expansions of $q$-character polynomials associated to unipotent matrices with various Jordan block sizes. An notable example is the product expansions of statistics counting single a Jordan block. 
\begin{example}[Example~\ref{example:Ja*Jb} below] \label{ex:intro-Ja*Jb}
For $b>a$, let $J_a$ and $J_b$ be unipotent Jordan blocks of sizes $a$ and $b$ respectively. Then, the product expansion are,
$$X_{J_a} \cdot X_{J_b} = X_{J_b}+ \left(\sum_{m=1}^{a-1} q^{2m-1}(q-1) X_{J_{b,m}}\right) + q^{2a} X_{J_{b,a}}$$
\end{example}

Product expansions of more complicated matrix statistics are listed throughout \S\ref{sec: expansion coefficients}, and many small cases are tabulated in Table~\ref{table:product expansions}.


There are several consequences and stability patterns manifested in the product expansions that we compute -- many already visible in Example \ref{ex:intro-Ja*Jb} above. In Subsection~\ref{subsec: conjectures} we formulate a number of explicit conjectures regarding the form and properties of the expansion polynomials $R_{\lambda, \mu}^{\nu}(q)$. One notable conjecture addresses the question of generators for the ring of $q$-character polynomials.
\begin{conjecture*}[\textbf{Single blocks generate the ring}]
Products of statistics $X_{J_a}$ associated to single Jordan blocks generate the entire ring of $q$-character polynomials over the field of rational functions $\mathbb{Q}(q)$.

More precisely, for every partition $\lambda = (\lambda_1\geq \ldots \geq \lambda_k)$ let $J_{\lambda}$ be the unipotent Jordan matrix with block sizes specified by $\lambda$. Then there exists a polynomial $P(t_1,\ldots,t_{\lambda_1})$ of degree $k$ with coefficients in $\mathbb{Q}(q)$ such that
\[
X_{J_\lambda} = P(X_{J_1},X_{J_2},\ldots,X_{J_{\lambda_1}}).
\]
\end{conjecture*}
Example \ref{example:Ja*Jb} shows that the conjecture holds for partitions with $\leq 2$ parts.

Finally, through the stabilization of expectations of $q$-character polynomials (\cite[Corollary 4.6]{GADISH2017450}), our product expansions facilitate a relatively straightforward approach to determining the higher and joint moments of the random variables $X_A$ for various $A$.
For example, it is easy to see that the number of eigenvectors with distinct eigenvalues are uncorrelated among random matrices of dimension 2 or greater.
A notable implication of the expansion given in the last example is the following correlation between the counts of subspaces on which a random matrix acts as a unipotent Jordan block of a given size -- these are the values of $X_{J_a}$.
\begin{theorem} \label{theorem: correlation}
For $b\geq a$, the correlation between random variables $X_{J_a}$ and $X_{J_b}$ in the uniform probability space $\Gl_n(\F_q)$ is $\sqrt{\frac{q^a-1}{q^b-1}}$ once $n \geq 2b$.
\end{theorem}
Similar explicit though more complicated formulas exist for smaller $n$. We will not discuss those here. For a discussion and proof of Theorem \ref{theorem: correlation}, see \S\ref{subsec: applications}.

We conclude this introduction by laying out some basic notation to be used in future sections. 
\begin{notation}
For a square matrix $S$, let $\dim(S)$ and $n_S$ both denote the number of rows of $S$.

Denote by $J_{n_1,n_2,\dots,n_i}(\lambda)$ the matrix in Jordan form with block sizes $n_1, n_2, \dots, n_i$ and eigenvalue $\lambda$. 
Alternatively, consider a partition $\mu$ of $n$ with frequency representation $(1^{m_1},2^{m_2},\dots,k^{m_k})$. Then let $J_{\mu}(\lambda)=J_{1^{m_1},2^{m_2},\dots,k^{m_k}}(\lambda)$ also denote the matrix in Jordan form with $m_i$ blocks of size $i$ and eigenvalue $\lambda$. The unipotent matrix $J_{\mu}(1)$ will be denoted simply by $J_{\mu}$ and the nilpotent matrix $J_{\mu}(0)$ by $N_{\mu}$.
Observe that the total number of rows of $J_{\mu}(\lambda)$ will always be $\|\mu\|=\sum im_i$.
\end{notation}

\subsection{Acknowledgements}
This project grew out of an MIT PRIMES project for high school students, in which the 2nd author mentored the remaining authors. We would like to thank the organizers of MIT PRIMES: Prof. Pavel Etingof, Slava Gerovitch, Tanya Khovanova and Alexander Vitanov for the opportunity to take part in this inspirational program, as well as their feedback and helpful comments.
We also thank Trevor Hyde for helpful correspondence and much needed perspective.

\section{$q$-character polynomial evaluations} \label{sec: reductions}
In order to determine the coefficients central to our problem, we will first show that for any two matrices in $\F_q$, the evaluation $X_A(B,q^d)$ is determined by evaluations with $A$ and $B$ unipotent matrices. Furthermore, we show that the latter evaluations on unipotents are polynomial in $t=p^d$ for any prime power. With these results, we compute the evaluations explicitly in Theorem~\ref{theorem: reducing to unipotents}.

\subsection{Disjoint sets of eigenvalues} 
First, we show that matrices with disjoint sets of eigenvalues interact trivially, in the following sense.

\begin{lemma}\label{lem:evaluations of disjoint eigenvalues}
Let $A$ and $B$ be square matrices over $\F_q$ and suppose that $B$ factors as direct sum of square matrices $B \sim B_1 \oplus B_2$ such that $B_2$ shares no eigenvalues with $A$, then $X_A(B) = X_A(B_1)$.
\end{lemma}
\begin{proof}
Let $\F^{\dim(B)} = V_1\oplus V_2$ be a decomposition such that $B$ acts on $V_i$ by $B_i$. Suppose $W\leq \F^{\dim(B)}$ is a space counted by $X_A(B)$, i.e. $B(W)\subseteq W$ and $B|_W \sim A$. This implies that $A$ and $B$ act on $W$ with the same eigenvalues. But since $B_2$ shares no such eigenvalues with $A$, it follows that $W\cap V_2 = 0$ and $W\subseteq V_1$. Therefore $W$ is counted in $X_A(B_1)$ as well.

Conversely, if $W\leq V_1$ is counted in $X_A(B_1)$, then since $W\leq V_1\oplus V_2 = \F^{\dim(B)}$ and satisfies $B|_W=B_1|_W$, it is clearly counted in $X_A(B)$ as well. Thus the two counts coincide.
\end{proof}

Similarly, consider the product of statistics.
\begin{lemma}\label{lem:product with disjoint eigenvalues}
If $A$ and $B$ are square matrices that share no eigenvalues, the product of statistics $X_A \cdot X_B$ is equal to $X_{\left(\substack{A 0\\0 B}\right)}$.
\end{lemma}

\begin{proof}
By \cite[Corollary 3.9]{GADISH2017450} there exist coefficients such that 
\[X_A\cdot X_B = \sum_{C} r_{A,B}^C X_C\]
where $C$ ranges over conjugacy classes of matrices where \[\max(\dim(A),\dim(B)) \leq \dim(C) \leq \dim(A)+\dim(B).\]
For every such $C$, the evaluation $X_A(C)$ counts the number of $\dim(A)$-dimensional subspaces of $\F^{\dim(C)}$ that are $C$-invariant and on which $C$ acts by $A$ up to conjugation. Let $S_A$ be the collection of these subspaces, so that $| S_A | = X_A(C)$. Similarly let $S_B$ be the set of the analogous subspaces for the case where $A$ is replaced by $B$.

Now we claim that for every $V_1\in S_A$ and $V_2\in S_B$ we have $V_1\cap V_2 = \{0\}$. Indeed, the intersection $W=V_1\cap V_2$ is itself $C$-invariant, since both $V_1$ and $V_2$ are such. Assume for the sake of contradiction that $\dim(W)>0$. Then the restriction of $C$ to $W$ has some eigenvalue $\lambda$ (perhaps only a member of a larger field). But since the action of $C$ on $V_1$ is conjugate to that of $A$, this $\lambda$ must then also be an eigenvalue $A$, and similarly it must also be an eigenvalue of $B$. This is a contradiction, as $A$ and $B$ have no common eigenvalues.

If the evaluation of $C$ on the product $X_A\cdot X_B$ is non-zero, then both $S_A$ and $S_B$ are non-empty. Thus from the above argument we conclude that for any choice of $V_1\in S_A$ and $V_2\in S_B$,
\[
\dim(A) + \dim(B) = \dim (V_1) + \dim (V_2) = \dim (V_1 \oplus V_2) \leq \dim(C)
\]
showing that the only matrices $C$ for which the product $X_A\cdot X_B$ is non-zero are of the maximal dimension $\dim(A)+\dim(B)$. It also follows that $\F^{\dim(C)} = V_1 \oplus V_2$, that is the space has a basis built from a basis for $V_1$ followed by a basis for $V_2$. Since both subspaces are $C$ invariant, in every such basis the matrix $C$ is represented by
\[
C' = \left(\substack{ A' 0 \\ 0 B'}\right)
\]
where $A' = C|_{V_1} \sim A$ and $B' = C|_{V_2} \sim B$, so $C'$ is conjugate to the block matrix $\left(\substack{ A 0 \\ 0 B}\right)$. But since change of basis corresponds to conjugation, it follows that $C\sim C'$. We thus found that the only matrices of dimension $\leq \dim(A)+\dim(B)$ that may evaluate on $X_A\cdot X_B$ nontrivially must be conjugate to the block matrix built from $A$ and $B$.

Conversely, this block matrix clearly evaluates to $1$ on both $X_A$ and $X_B$, thus giving the desired equality
\[
X_A\cdot X_B = X_{\left(\substack{ A 0 \\ 0 B}\right)}.
\]
\end{proof}

\subsection{Independence of eigenvalues}
Now that it has been shown that is possible to consider matrices with only one eigenvalue, the next reduction verifies that the choice of eigenvalue for our matrices does not matter.

\begin{lemma}\label{lem: changing eigenvalue}
Let $A_0$ and $B_0$ be two nilpotent matrices. For every two scalars $\lambda, \lambda' \in \F$ denote $A_{\lambda}=\lambda I + A_0$ and $A_{\lambda'}=\lambda' I+A_0$, similarly denote $B_{\lambda}$ and $B_{\lambda'}$. Observe that $A_{\lambda}$ and $A_{\lambda'}$ have the same Jordan form and generalized eigenvectors but different eigenvalues. Then, we have $X_{A_{\lambda}}(B_{\lambda})=X_{A_{\lambda'}}(B_{\lambda'})$.
\end{lemma}
\begin{proof}
It is easy to see that a subspace $W$ is $B_{\lambda}$-invariant iff if it is $B_{\lambda'}$-invariant as $B_{\lambda}(W)\subseteq W$ iff $B_{\lambda}(W) + (\lambda-\lambda')W\subseteq W$. It is similarly clear that the restriction $B_\lambda|_W$ is conjugate to $A_\lambda$ iff the restriction $B_{\lambda'}|_W$ is conjugate to $A_{\lambda'}$. 

Thus it follows that a space $W$ is $B_{\lambda}$-invariant and $B_{\lambda}|_W\sim A_{\lambda}$ if and only if $W$ is $B_{\lambda'}$-invariant and $B_{\lambda'}|_W\sim A_{\lambda'}$. Hence the sets of all such $W$ coincide, which by definition gives $X_{A_\lambda}(B_\lambda)=X_{A_\lambda'}(B_{\lambda'}).$
\end{proof}
Since the statistics are invariant under changing eigenvalues, one can, in particular, reduce general calculations to those involving only unipotent matrices, i.e. matrices in which all eigenvalues are 1.

\subsection{Field extensions and Jordan matrices}
Next, we determine how the evaluations of matrices are affected by field extensions, ultimately facilitating the reduction of any evaluation and expansion calculations to those involving statistics on matrices in Jordan form.

We begin by considering the following simple example. 
\begin{example}\label{ex:field extension example}
Recall that over the field $\F_4$ there is a similarity $\begin{pmatrix} 1 & 1 \\ 1 & 0 \end{pmatrix} \sim \begin{pmatrix} \epsilon & 0 \\ 0 & \epsilon^2 \end{pmatrix}$ where $\epsilon \in \F_4$ satisfies $\epsilon^2+\epsilon+1=0$. Our next theorem shows that for any matrix $B$ over $\F_2$ we have $X_{\left(\substack{ 11 \\ 10}\right)}(B,2) = X_{\left( \epsilon \right)}(B,4)$.
\end{example}

We let $F:=\Fr_q$ denote the Frobenius automorphism of $\F_q$.
\begin{lemma}\label{lemma: specific field extension} For $M_0$ a nilpotent matrix in Jordan form, denote $M_{\mu} = M_0 + \mu I$ for every scalar $\mu$. Let $A$ be a matrix with entries in $\F_q$ that is conjugate to
\[ \begin{pmatrix} M_{\lambda} & & &  \\ & M_{F(\lambda)} & &  \\ & & \ddots &  \\ & & & M_{F^{d-1}(\lambda)}  \end{pmatrix} \]
for $d>1$ such that $\F_q[\lambda] = \F_{q^d}$.

Let $B$ be an $n\times n$ matrix with coefficients in $\F_q$. Then,
$X_{A}(B,q) = X_{M_{\lambda}}(B,q^d).$ 
\end{lemma}

\begin{proof}
Let $S = \{ W \leq \F_q^n \mid W \text{ is } B\text{-invariant and } B|_W \sim A \}$ and $T =  \{ V \leq \F_{q^d}^n \mid V \text{ is } B\text{-invariant}\\ \text{and } B|_V \sim M_{\lambda} \}$. We wish to construct a bijection between $S$ and $T$ since $|S| = X_{A}(B,q)$ and $|T| = X_{M_{\lambda}}(B,q^d)$.

Consider $W \in S$. Extend scalars to $\F_{q^d}$ to form $\overline{W}$. It is clear that $\overline{W}$ is still $B$-invariant and $B$ acts by $A$ up to conjugation on $\overline{W}$. Now, consider the function $g:S\to T$ given by sending $W$ to the $\lambda$-generalized eigenspace of $B|_{\overline{W}}$. This map produces a subspace that must be $B$-invariant, as it is a subspace of $\overline{W}$, and on which $B$ acts by $M_{\lambda}$ up to conjugation.

Conversely, consider $V \in T$. As $V$ is $B$-invariant, $Bv \in V$ for all $v \in V$, so $F^i(Bv) = B F^i(v) \in F^i(V)$ for all $i$. Thus $F^i(V)$ is $B$-invariant as well.
Furthermore, for any sequence of vectors $v_1,\ldots,v_k$ such that $Bv_i = \lambda v_i + v_{i-1}$ (with $v_0=0$), one has that $B(F(v_i)) = F(B(v_i)) = F(\lambda)F(v_i) + F(v_{i-1})$. Thus, on $F(V)$ the operator $B$ acts through a matrix conjugate to $M_{F(\lambda)}$.

Let $\overline{W} = \bigoplus_{i=0}^{d-1} F^i(V)$. Since every $F^i(V)$ is $B$-invariant, the sum is also. Furthermore, since $F: F^i(V)\to F^{i+1}(V)$, the sum is also $F$-invariant. Lastly, $B$ acts on $\overline{W}$ by the block matrix containing the Jordan matrices $M_{F^i(\lambda)}$ for $0 \leq i \leq d-1$, thus the $B$ action on $\overline{W}$ is conjugate to $A$ by assumption. By Galois descent (see e.g. \cite{conrad}), there exists a unique $W \leq \F_q^n$ whose extension to $\F_{q^d}^n$ is precisely $\overline{W}$. Explicitly, $W$ consists of the $F$-fixed vectors in $\overline{W}$.

As $\overline{W}$ is $B$-invariant, for all $w \in W \subset \overline{W}$ we have $Bw \in \overline{W}$. However, we also have $F(Bw)  = F(B)F(w) = Bw$ because $W$ consists of the $F$-fixed vectors. Thus it follows that $Bw$ is $F$-fixed, and so we must have $Bw \in W$, i.e. $W$ is $B$-invariant. Also, $B$ acts by $A$ up to conjugation on $W$ because $A$ and $B|_W$ have the same rational canonical form. Thus, $W \in S$, so we have constructed a map from $h:T \to S$ that sends $V$ to $W$. We now show that $g$ and $h$ are inverses, thus establishing $|S|=|T|$ and completing the proof.

In the construction of $g(h(V)) = g(W)$, one first considers $\overline{W}$ after extending scalars. We then consider the map from $\overline{W}$ to the $\lambda$-generalized eigenspace of $B|_{\overline{W}}$, which is $V$ because $V$ is a summand of $\overline{W}$ and on all other summands, $B$ has different eigenvalues. We conclude that $g(h(V)) = V$.

Now, consider $h(g(W'))$ for some $W' \in S$, and let $\overline{W'}$ be the space formed by extending the scalars of $W'$ to $\F_{q^d}$. We claim that the $F^{i}(\lambda)$-generalized eigenspace of $B|_{\overline{W'}}$ is precisely $F^i(g(W'))$. Note that for all $v' \in g(W')$ we have $(B-\lambda I)^m v'=0$ for some fixed $m$. Then, $F((B-\lambda I)^m v') = (B - F^i(\lambda) I)^m F^i(v') = 0$, so $F^i(g(W'))$ must be equal to the $F^i(\lambda)$-generalized eigenspace. Thus, we observe that the map $h$ on $g(W')$ precisely takes a direct sum of the eigenspaces of $\overline{W'}$ to produce $\overline{W'}$. Finally, $h$ reduces $\overline{W'}$ back to $W'$ by uniqueness of Galois descent. Therefore, $h(g(W'))=W'$. 

This implies $g$ and $h$ are bijections, so $|S|=|T|$.
\end{proof}

Thus, this lemma demonstrates that determining the evaluations of statistics on matrices in Jordan form will prove sufficient in computing general evaluations.

\subsection{Reducing to unipotents}
With the reductions given above, one can express general evaluations $X_A(B,q^d)$ in terms of evaluations of statistics on unipotent matrices. We compute those evaluations explicitly in the next section (see Theorem \ref{theorem:general evaluations} below), but let us first use the fact of their polynomiality to prove Theorem \ref{theorem-main:polynomiality}.


\begin{lemma}\label{field_extension_evaluation}
 For $M_0$ a nilpotent matrix in Jordan form, denote $M_{\mu} = M_0 + \mu I$ for every scalar $\mu$. Let $A$ be a matrix with entries in $\F_q$ that is conjugate over $\F_{q^d}$ to
\[ \begin{pmatrix} M_{\lambda} & & &  \\ & M_{F(\lambda)} & &  \\ & & \ddots &  \\ & & & M_{F^{d-1}(\lambda)}  \end{pmatrix} \]
for $d>1$ such that $\F_q[\lambda] = \F_{q^d}$. Let $B$ be an $n\times n$ matrix with coefficients in $\F_q$. Then $\forall m\geq 1$,
$$X_A(B,q^m) = X_{M_{\lambda}} (B, q^{\lcm(m,d)})^{\gcd(m,d)}.$$
\end{lemma}
\begin{proof}
We will use the matrices $\overline{N}_i = \bigoplus_{j=0}^{d/\gcd(m,d)-1} M_{F^{i+mj}(\lambda)}$ where $0 \leq i \leq \gcd(m,d)-1$. 

Let $n(A)$ denote $\dim(A)$, and let $V \leq \mathbb{F}_{q^{\lcm(m,d)}}^{n(A)}$ be a vector subspace that is $A$-invariant and $T|_{A} \sim M_{F^{i}(\lambda)}$ (such vector spaces exist since $\F_{q^d} \leq \F_{q^{\lcm(m,d)}}$). Then as indicated in the proof of Lemma \ref{lemma: specific field extension}, the image $F^{mj}(V)$ is similarly $A$-invariant and $A|_{F^{mj}(V)} \sim M_{F^{i+mj}(\lambda)}$.

Now consider the sum $\overline{W} = \bigoplus_{j=0}^{d/\gcd(m,d)-1} F^{i+mj}(V)$, and observe that $A$ acts on 
$\overline{W}$ by the block matrix $\overline{N}_i$.
Then $\overline{W}$ is clearly $F^m$-invariant. By Galois descent (see e.g. \cite{conrad}) there exists a unique $W \leq \F_{q^m}^{n(A)}$ such that its extension to $\F_{q^{\lcm(m,d)}}^{n(A)}$ is $\overline{W}$. 
Observe that $\overline{W}$ is $A$-invariant as a sum of such subspaces. Then by uniqueness of $W$ it follows that it is also $A$-invariant. It follows that $A|_{W}$ is defined over $\F_{q^m}$ and conjugate to $\overline{N}_i$, and in particular $\overline{N}_i$ is conjugate to a matrix $N_i$ defined over $\F_{q^m}$.

Note that $A = \bigoplus_{i=0}^{\gcd(m,d)-1} N_i$. Since the $\{N_i\}_{i=0}^{\gcd(m,d)-1}$ have pairwise disjoint sets of eigenvalues, by Lemma~\ref{lem:product with disjoint eigenvalues},
\[ X_A(B, q^m) = \prod_{i=0}^{\gcd(m,d)-1} X_{N_i} (B, q^m) .\]
Let $t=q^m$. The smallest power $\ell$ for which $\F_{t^\ell}$ contains the eigenvalues of $A$ is $\lcm(m,d)/m$, so by Lemma~\ref{lemma: specific field extension}, \[ X_{N_i}(B, t) = X_{N_i} (B, t^{\lcm(m,d)/m}) = X_{M_{F^i(\lambda)}} (B, q^{\lcm(m,d)}) \]
where the second equality is apparent from the construction $\overline{N}_i = \bigoplus_{j=0}^{d/\gcd(m,d)-1} M_{F^{i+mj}(\lambda)}$. 

Now, in the Jordan normal form of $B$ over $\F_{q^{\lcm(m,d)}}$, let $B_{F^i(\lambda)}$ denote the sum of blocks corresponding to eigenvalue $F^i(\lambda)$. Note that $B_{F^{i}(\lambda)}$ and $B_{F^j(\lambda)}$ have the same Jordan form since $B$ was defined over $\F_{q}$ and $F^i(\lambda)$ and $F^j(\lambda)$ are Galois conjugate. Therefore, by Lemma~\ref{lem:evaluations of disjoint eigenvalues} and Lemma~\ref{lem: changing eigenvalue},
\[ X_{M_{F^i(\lambda)}} (B, q^{\lcm(m,d)}) = X_{M_{F^i(\lambda)}} (B_{F^i(\lambda)}, q^{\lcm(m,d)}) = X_{M_{\lambda}} (B_{\lambda}, q^{\lcm(m,d)}) = X_{M_{\lambda}} (B, q^{\lcm(m,d)}).\]
Finally, we have,
\[ X_{A}(B, q^m) = \prod_{i=0}^{\gcd(m,d)-1} X_{N_i} (B, q^m) = \prod_{i=0}^{\gcd(m,d)-1} X_{M_{\lambda}} (B, q^{\lcm(m,d)}) = X_{M_{\lambda}} (B, q^{\lcm(m,d)})^{\gcd(m,d)}. \]
\end{proof}

\begin{theorem}\label{theorem: reducing to unipotents}
For a pair of invertible matrices $A$ and $B$ over $\F_q$, let $\zeta_1,\ldots,\zeta_k$ be a set of representatives for the Galois-conjugacy classes of eigenvalues of $A$. For every $1\leq i\leq k$ let $\lambda_i$ and $\mu_i$ denote the partitions enumerating the Jordan block sizes with eigenvalue $\zeta_i$ in $A$ and $B$, respectively ($mu_i$ is the zero partition if $\zeta_i$ is not an eigenvalue of $B$). Then,
\[ P_{A,B} (q^m) = \prod_i  P_{\lambda_i, \mu_i} (q^{\lcm(m,d_i)})^{\gcd(m,d_i)} \]
where $d_i$ denotes the degree of the minimal polynomial of $\zeta_i$ over $\F_q$.
\end{theorem}
\begin{proof}[Proof of Theorems \ref{theorem: reducing to unipotents} and \ref{theorem-main:polynomiality}]

First, let $f_A(q^m)$ be the characteristic polynomial of $A$ and factor $f_A$ over $\F_q[x]$ into distinct irreducible polynomials $f_1, f_2, \ldots, f_k$ with multiplicities $r_1, \ldots, r_k$. That is,
\[ f_A (q^m) = f_1(q^m)^{r_1} f_2 (q^m)^{r_2} \cdots f_k (q^m)^{r_k} . \]
Let $A_i$ be the restriction of $A$ to the subspace $\ker (f_i(A)^{r_i})$. Observe that the characteristic polynomial of $A_i$ is $f_i(q^m)^{r_i}$ and $A \sim \bigoplus_{i=1}^k A_i$. Since for every $i\neq j$ the matrices $A_i$ and $A_j$ have coprime characteristic polynomials, they have disjoint sets of eigenvalues. So by Proposition~\ref{lem:product with disjoint eigenvalues}, we have $X_A(B) = \prod_{i=1}^{k} X_{A_i} (B)$.

The characteristic polynomial of $A_i$ can be factored over a field extension as 
$$f_i(q^m)^{r_i} = (q^m-\zeta_i)^{r_i}(q^m-\Fr(\zeta_i))^{r_i}\cdots(q^m -\Fr^{d_i-1}(\zeta_i))^{r_i}$$ where $d_i = \deg(f_i)$ and $\zeta_i \in \F_{q^{d_i}}$. Furthermore, $A_i \sim \bigoplus_{j=0}^{d_i-1} \Fr^j(M_{\zeta_i})$ where $M_{\zeta_i}$ is an $r_i \times r_i$ matrix in Jordan normal form with eigenvalue $\zeta_i$. Therefore, by Lemma~\ref{field_extension_evaluation}, $X_{A_i}(B,q^m) = X_{M_{\zeta_i}}(B, q^{\lcm(m,d_i)})^{\gcd(m,d_i)}$.

Decompose $B$ as $N_{\zeta_i}\oplus N'$ where $N_{\zeta_i}$ is in Jordan form over $\F_{q^{d_i}}$ and has eigenvalue $\zeta_i$, and $N'$ has eigenvalues distinct from $\zeta_i$. This decomposition is possible since $\zeta_i \in \F_{q^{d_i}}$. Then, by Lemma~\ref{lem:evaluations of disjoint eigenvalues}
$$X_{M_{\zeta_i}}(B, q^{\lcm(m,d_i)})^{\gcd(m,d_i)}=X_{M_{\zeta_i}}(N_{\zeta_i}\oplus N', q^{\lcm(m,d_i)})^{\gcd(m,d_i)} = X_{M_{\zeta_i}}(N_{\zeta_i}, q^{\lcm(m,d_i)})^{\gcd(m,d_i)}.$$ 
Let $\lambda_i$ and $\mu_i$ be the partitions that represent the Jordan block sizes of $M_{\zeta_i}$ and $N_{\zeta_i}$, respectively. By Lemma~\ref{lem: changing eigenvalue}, $X_{M_{\zeta_i}}(N_{\zeta_i}, q^{\lcm(m,d_i)}) = P_{\lambda_i, \mu_i} (q^{\lcm(m,d_i)})$. Altogether,
\begin{equation}\label{eq:polynomials as product of unipotents}
    X_A(B,q^m) = \prod_{i=1}^{k} X_{A_i} (B,q^m) =\prod_{i=1}^k  P_{\lambda_i, \mu_i} (q^{\lcm(m,d_i)})^{\gcd(m,d_i)}
\end{equation}
as claimed. 

Note that this formula for $X_A(B,q^m)$ is by definition periodically polynomial as specified in Theorem \ref{theorem-main:polynomiality}. 
\end{proof}
\begin{corollary}
When all eigenvalues of $A$ lie in $\F_q$, the evaluation $X_A(B,q^m)$ is polynomial in $q^{m}$.
\end{corollary}
\begin{remark}
A special case of this corollary, when $A$ and $B$ are unipotent matrices in Jordan normal form, implies that there is a single polynomial $P_{\lambda,\mu}$ which coincides with $X_{J_\lambda}(J_\mu,q^m)$ for all prime powers $q^m$.
\end{remark}

\section{Calculating unipotent evaluation polynomials $P_{\lambda,\mu}$}\label{sec: evaluations}
Having shown that it suffices to consider unipotent matrics, we now explicitly determine the evaluations of our statistics on unipotent matrices. A unipotent conjugacy class is characterized by an integer partition, enumerating the sizes of the Jordan blocks. Hence, if for every partition $\lambda$ we let $J_\lambda$ denote a unipotent Jordan matrix with block sized specified by $\lambda$, the evaluations $P_{\lambda,\mu}(q):= X_{J_\lambda}(J_\mu,q)$ are parameterized by pairs of partitions. In this section, we will calculate $P_{\lambda,\mu}(q)$, showing in the process that they are polynomial in $q$ with positive integer coefficients.
To begin, recall the definition of the $q$-binomial coefficient.
\begin{definition}
The $q$-binomial coefficient $\binom{m}{n}_q$ counts the number of $n$-dimensional subspaces in $\F_q^m$. Explicitly,  
$$\binom{m}{n}_q=\frac{(q^m-1)(q^m-q)(q^{m}-q^2)\cdots (q^m-q^{n-1})}{(q^n-1)(q^n-q)\cdots (q^n-q^{n-1})}=\prod_{i=0}^{n-1}\frac{q^m-q^i}{q^n-q^i}$$ 
if $m\geq n$ and is $0$ otherwise.
\end{definition} 
It is a fact that the $q$-binomial coefficients are polynomials in $q$ with positive integer coefficients (see e.g. \cite{q-binomial}).

We start by considering the evaluation of statistics associated with matrices composed of Jordan blocks of equal size.

\begin{lemma} \label{lemma:equal blocks}
Fix matrices $A=J_{j^{a_j}}$ and $B=J_{1^{b_1},\ldots,k^{b_k}}$ and let $n_B$ denote the number of rows in $B$. Then,
$$X_A(B)= \binom{\sum_{i=j}^k b_i}{a_j}_q \cdot q^{a_j\left[ n_B - \sum_{i=j}^k (i-j+1)b_i - (j-1)a_j  \right]}$$
\end{lemma}
\begin{proof} 
Denote the ambient space on which $B$ acts by $V=\F_q^{n_B}$. We need to count the number of $B$-invariant subspaces $W_j \leq V$ that have $B|_{W_j} \sim J_{j^{a_j}}$.
With respect to $N:=B-I$ -- the nilpotent matrix with the same Jordan blocks as $B$ -- the counting problem amounts to computing the cardinality of the set
$$K_j=\{ W_j<V \mid N(W_j)\subset W_j, N|_{W_j}\sim N_{j^{a_j}}\}.$$
To do this, consider following quotients:

$$
\xymatrix{ \ker(N^{j}) \ar[d]^\pi & \subseteq & V \ar[d]^\pi \\
\ker(N^{j})/\ker(N^{{j-1}}) & \subseteq & V/\ker(N^{{j-1}})}
$$

We start by showing that the quotient map $\pi$ induces a surjection from $K_j$ to the set of $a_j$-dimensional subspaces of $\ker(N^{j})/\ker(N^{j-1})$. Then we proceed by observing that all fibers of this surjection have equal cardinality, hence the count will be given by $|K_j| = |\text{base}|\times|\text{fiber}|$.

\begin{claim} \label{claim: well-defined} 
For every $W_j\in K_j$ we have that $\overline{W}_j:=\pi(W_j)$ is a $a_j$-dimensional subspace of $\ker(N^{j})/\ker(N^{{j-1}})$ 
\end{claim}
\begin{proof}
Fix $W_j\in K_j$. Then since $ N|_{W_j}\sim N_{j^{a_j}}$, i.e. similar to a nilpotent matrix with $a_j$ Jordan blocks of size $j$, it follows that $W_j$ has a basis of the form 
\begin{align*}
    &\{v_1, Nv_1, N^2v_1, \dots, N^{j-1}v_1\\
    &v_2, Nv_2, N^2v_2, \dots, N^{j-1}v_2\\
    &\vdots \hspace{1cm}\vdots \hspace{1cm}\vdots\hspace{1cm} \vdots \\
    &v_{a_j}, Nv_{a_j}, N^2v_{a_j}, \dots, N^{j-1}v_{a_j}\}\\
\end{align*}
such that $N^{j}v_i=0$ for all $i$. Thus $\overline{W}_j \leq \ker(N^{j})$. It further follows that $\overline{W}_j$ is spanned by $\langle \pi (v_1), \pi (v_2), ... , \pi (v_{a_j})\rangle$ and is of dimension at most $a_j$. To see that the $\pi(v_i)$'s are indeed independent suppose
\begin{align*}
    0&=\lambda_1\pi(v_1)+\lambda_2\pi(v_2)+... +\lambda_{a_j}\pi(v_{a_j}) = \pi(\lambda_1v_1+\lambda_2v_2+...+\lambda_{a_j}v_{a_j}).
\end{align*}
It follows that $\lambda_1v_1+\lambda_2v_2+...+\lambda_{a_j}v_{a_j}\in \ker (N^{j-1})$. So, \begin{align*} 
    0&=N^{j-1}(\lambda_1v_1+\lambda_2v_2+...+\lambda_{a_j}v_{a_j})\\
    &=\lambda_1 N^{j-1}v_1+\lambda_2 N^{j-1}v_2+\dots+ \lambda_{a_j} N^{j-1}v_{a_j}.
\end{align*}  
But $\{N^{j-1}v_1, N^{j-1}v_2, \dots, N^{j-1}v_{a_j}\}$ is a subset of a basis of $W_j$, thus $\lambda_1=\lambda_2=\dots=\lambda_{a_j}=0$. It follows that $\pi(v_1), \pi(v_2), ... \pi(v_{a_j})$ are linearly independent and $\overline{W}_j$ is of dimension $a_j$. 
\end{proof}
\begin{claim}\label{claim:pi is surjective} For every $a_j$-dimensional subspace $\overline{W}_j \leq \ker(N^{j})/\ker(N^{j-1})$ there exists $W_j \in K_j$ such that $\pi(W_j)=\overline{W}_j$
\end{claim}
\begin{proof}
Let $\overline{W}_j$ be any $a_j$-dimensional subspace of $\ker(N^{j})/\ker(N^{j-1})$ and pick $\{v_1,\ldots, v_{a_j}\} \subset \ker(N^{j})$ such that $\{\pi(v_1), \pi(v_2), \dots, \pi(v_{a_j})\}$ form a basis for $\overline{W}_j$.
Consider the space 
\begin{align*}
W_j=\Span(&v_1, Nv_1, N^2v_1,\dots, N^{j-1}v_1,\\ & v_2, Nv_2, N^2v_2, \dots, N^{j-1}v_2,\\ & \vdots \hspace{1cm}\vdots \hspace{1cm} \vdots \\ &v_{a_j}, Nv_{a_j}, N^2v_{a_j}, \dots, N^{j-1}v_{a_j}). 
\end{align*}
We claim that $W_j\in K_j$ and $\pi(W_j)=\overline{W}_j$. Indeed, since all $N^r v_i \in \ker(N^{j-1})$ for $r \geq 1$ vanish in the quotient, it follows that $\pi(W_j) = \overline{W}_j$.

To see that $W_j\in K_j$ it remains to show that $N(W_j)\subseteq W_j$ and $N|_{W_j}\sim N_{j^{a_j}}$. The $N$-invariance is clear since it obviously holds on the spanning set defining $W_j$. To determine the conjugacy class of $N|_{W_j}$, we only need to show that the spanning set is linearly independent, since this would imply that $\{ N^r v_i \}_{i,r}$ is a basis with respect to which $N|_{W_j}$ has the desired Jordan form.

Suppose that 
\begin{align*}
&\lambda_{1,1}v_1 + \lambda_{1,2}Nv_1 + \lambda_{1,3}N^2v_1 + \dots + \lambda_{1,j}N^{j-1}v_1 \\ 
+& \lambda_{2,1}v_2 + \lambda_{2,2}Nv_2 +  \lambda_{2,3}N^2v_2 +  \dots +  \lambda_{2,j}N^{j-1}v_2\\ & \vdots \hspace{1cm}\vdots \hspace{1cm} \vdots \hspace{1cm} \vdots \hspace{1cm}\vdots \hspace{1cm} \vdots \\ +& \lambda_{a_j,1}v_{a_j} +  \lambda_{a_j,2}Nv_{a_j} + \lambda_{a_j,3}N^2v_{a_j}, \dots + \lambda_{a_j,j}N^{j-1}v_{a_j}=0.
\end{align*}
Applying $\pi$ to this equation gives $$\lambda_{1,1}\pi(v_1)+\lambda_{2,1}\pi(v_2)+\lambda_{3,1}\pi(v_3)+\dots \lambda_{a_j,1}\pi(v_{a_j})=0.$$ But $\{\pi(v_1),\pi(v_2), \dots, \pi(v_{a_j})\}$ was chosen to be a basis for $\overline{W}_j$ so $\lambda_{1,1}=\lambda_{2,1}=\dots=\lambda_{a_j,1}=0$. Continue by induction to show that the first $r$ columns of the equation above vanish.

Assume that $\lambda_{i,s}=0$ for all $i$ and $s < r$, so the linear relation reduces to $$N^r \left( \sum_{i=1}^{a_j}\sum_{s=r}^{j-1} \lambda_{i,s}N^{s-r}v_i  \right) = 0.$$
This implies that the argument $\sum\sum \lambda_{i,s}N^{s-r}v_i$ already vanishes in the quotient. Applying $\pi$ to this argument gives a new relation
$$ \lambda_{1,r}\pi(v_1)+\ldots + \lambda_{a_j,r} \pi(v_{a_j}) = 0 $$
which again by linear independence of the $\pi(v_i)$'s gives $\lambda_{1,r}=\ldots=\lambda_{a_j,r}=0$, thus completing the induction step.

Thus the set $\{N^{j}v_i\}_{i,j}$ is indeed a basis, so $W_j\in K_j$ and $\pi(W_1)=\overline{W}_1$ as claimed.
\end{proof}

Note the following useful fact: for every $r\geq 0$
\begin{equation}\label{eq:dim ker}
\dim(\ker N^{r}) = n_B - \dim (\operatorname{im} N^r) = n_B - \sum_{i\geq r} (i-r)b_i
\end{equation}
where $b_i$ is the number of blocks of size $i$ in the Jordan matrix $N$. In particular, it follows that there are 
\begin{align*}
        \binom{\dim(\ker N^{j})- \dim(\ker N^{j-1})}{a_j}_q&= \binom{\sum_{i\geq j}b_i}{a_j}_q
\end{align*}
ways to pick a $a_j$-dimensional subspace $\overline{W}_j \leq \ker(N^{j})/\ker(N^{j-1})$.

We now turn to computing the number of $W_j\in K_j$ that project to a given quotient $\overline{W}_j$.
\begin{claim}
Fix a $a_j$-dimensional subspace $\overline{W}_j \leq \ker(N^{j})/\ker(N^{j-1})$.
There are precisely
$$q^{a_j\left[ n_B - \sum_{i\geq j} (i-j+1)b_i - (j-1)a_j \right]}$$
spaces $W_j \in K_j$ such that $\pi(W_j)=\overline{W}_j$.
\end{claim}

\begin{proof}
If $s:\overline{W}_j \to \ker(N^{j})$ is any section of $\pi$, with image $W_0$, then the proof of Claim \ref{claim:pi is surjective} shows that $W_j := \sum_r N^r(W_0)$ is a preimage of $\overline{W}_j$ belonging to $K_j$. So we begin by counting sections of $\pi: \pi^{-1}(\overline{W}_j)\to \overline{W}_j$. 

Fix a section $s_0: \overline{W}_j \to W_0$. Then the set of all sections is in bijection with $\hom(\overline{W}_j,\ker(N^{j-1}))$
where a map $\phi:\overline{W}_j \to \ker(N^{j-1}) $ corresponds to the section $s_0+\phi$. It follows that the number of sections is
$$q^{\dim \overline{W}_j \cdot \dim \ker(N^{j-1})} = q^{a_j(n_B - \sum_{i\geq j}(i-j+1)b_i)}.$$
However, in going from sections $s$ to preimages $W_j\in K_j$ there is some overcounting, which we now address. Fix some $W_j\in K_j$ such that $\pi(W_k)=\overline{W}_j$, we wish to count the number of sections of $\pi:W_j\to \overline{W}_j$. The same argument as in the previous paragraph shows that the number such sections is
$$
\lvert\hom( \overline{W}_j, W_j\cap \ker(N^{j-1}) ) \rvert = q^{a_j\cdot \dim (W_j\cap \ker(N^{j-1}))}.
$$
But because $N|_{W_j}\sim N_{j^{a_j}}$, it follows that $\dim (W_j\cap \ker(N^{j-1})) = (j-1)a_j$. This is the overcounting involved in counting sections, so we divide by this amount to get the claimed number of preimages.
\end{proof}

Combining all claims in the proof, the set $K_j$ of interest maps onto a set of size $\binom{\sum_{i\geq j} b_i}{a_j}_q$ with fibers of equal size given by the previous claim. It follows that
$$X_A(B)= |K_j| = \binom{\sum_{i\geq j} b_i}{a_j}_q \cdot q^{a_j\left[ n_B - \sum_{i\geq j}(i-j+1)b_i - (j-1)a_j\right]}.$$
This completes the proof of Lemma~\ref{lemma:equal blocks}.
\end{proof}

Using Lemma~\ref{lemma:equal blocks}, we now prove the general case by induction. 

\begin{theorem}[General Evaluation Formula]\label{theorem:general evaluations}
    Consider the matrices $A = J_{1^{a_1}2^{a_2},\ldots,k^{a_k}}$ and $B=J_{1^{b_1}2^{b_2},\ldots,\ell^{b_\ell}}$. Then,
\begin{equation}\label{eq:evaluation}
    X_A(B) = \prod_{j=1}^k \binom{\sum\limits_{i>j} (b_i-a_i) + b_j }{a_j}_q \cdot q^{ a_j\left( n_B - \sum\limits_{i\geq j}(i-j+1)b_i -(j-1)a_j - 2(j-1)\sum\limits_{i>j}a_i - \sum\limits_{i<j} a_i\right)}.
\end{equation} 
    In particular, only the number of blocks in $B$ of size $\geq k$ enters the formula, and not their specific sizes.
\end{theorem}

\begin{proof}
We prove this theorem by induction on $k$ -- the largest block size of $A$.

The base case is when there are 0 blocks in $A$, corresponding to the empty matrix acting on the zero vector space. The only vector space counted in this case is the zero space, so the count $X_A(B)=1$, as is the RHS of \eqref{eq:evaluation}.

Assume by induction that the formula is true for matrices $A'$ with block sizes smaller than $k$. Let $A=J_{1^{a_1}\ldots k^{a_k}}$ be a matrix with block sizes of at most $k$.

As in the proof of Lemma~\ref{lemma:equal blocks}, we set $N:= B-I$ for the nilpotent matrix.
Now, since $(A-I)^{k}=0$ it follows that any $B$-invariant space $W$ on which $B|_{W}\sim A$ has $N^{k}(W)=(B-I)^{k}(W)=0$. So for the purpose of counting subspaces $W$ of this form, it is sufficient to restrict the ambient space on which $B$ acts to $\ker(N^{k})$, which we denote by $V$. On this smaller ambient space, the transformations $B$ and $N$ have their Jordan blocks restricted to have size at most $k$. Thus, without loss of generality we assume that $B$ has only blocks of size at most $k$. Note that the terms $n_B-\sum_{i\geq j}(i-j+1)b_i$ remain unchanged by this replacement for all $j\leq k$ since they are measuring $\dim(\ker N^{j-1})$ by \eqref{eq:dim ker}. This in particular shows that the sizes of blocks larger than $k$ does not matter, and rather only their total count is significant.

By Lemma~\ref{lemma:equal blocks}, we know that the number of $B$-invariant subspaces on which $B$ acts by a transformation conjugate to $J_{k^{a_k}}$ is
\begin{equation} \label{eq:max block count}\binom{\sum_{i\geq k} b_i}{a_k}_q \cdot q^{a_k\left[ n_B - \sum_{i\geq k}(i-k+1)b_i - (k-1)a_k\right]}.\end{equation}
Let $W_k$ be any subspace of this form. We wish to count the number of way to extend $W_k$ in an $B$-invariant way so that $B$ acts by a transformation conjugate to $A$. This is equivalent to finding a subspace $W' \leq V/W_k$ that is $B$-invariant and on which $B$ is conjugate to $A':= J_{1^{a_1}\ldots (k-1)^{a_{k-1}}}$.
Since this latter matrix has blocks sizes smaller than $k$, our inductive hypothesis applies, and the number of such subspaces is known. However, note that the transformation induced by $B$ on the quotient, call it $B'$, is represented by a matrix obtained from $B$ by removing $a_k$ blocks of the maximal size $k$. That is, $B'$ has $b'_k=b_k-a_k$ and $b'_j=b_j$ for all $j\neq k$. Also note that $n_{B'} = n_B - ka_k$.

Therefore, the number of ways to pick the subspace $W' \leq V/W_k$ is
\begin{equation}\label{eq:small block count}\prod_{j=1}^{k-1} \binom{\sum\limits_{i>j}(b_i-a_i) + b_j}{a_j}_q \cdot q^{a_j\left[ n_B-ka_k-\sum\limits_{i>j}(i-j+1)b_i + (k-j+1)a_k -(j-1)a_j -2(j-1)\sum\limits_{i=j}^{k-1} a_i - \sum\limits_{i<j}a_i \right]}.\end{equation}
The preimage subspace $W$ for which $W/W_k = W'$ is uniquely determined by $W'$, and is an $B$-invariant subspace satisfying $B|_{W}\sim A$.

We thus have a count of the number of pairs $(W_k, W)$ such that $W_k\leq W$ are two $B$-invariant spaces with respective restrictions of $B$ conjugate to $J_{k^{a_k}}$ and $A$: this is the product of \eqref{eq:max block count} and \eqref{eq:small block count}. However, we are interested in counting only the set of subspaces $W$ alone, so we must divide by the number of pairs $(W_k,W)$ with a given space $W$.

Fixing $W$, the number of choices for a subspace $W_k\leq W$ that is $B$-invariant and on which $B$ acts as $J_{k^{a_k}}$ is again counted in Lemma~\ref{lemma:equal blocks}. Since $B$ acts on $W$ as the matrix $A$ we have $a_i=b_i$ for all $i$ in the lemma, thus the number of such subspaces $W_k\leq W$ is
$$\binom{a_k}{a_k}_q \cdot q^{a_k(n_A-a_k - (k-1)a_k)} = 1\cdot q^{a_k\sum\limits_{j=1}^{k-1} j a_j}.$$
Dividing the product of \eqref{eq:max block count} and \eqref{eq:small block count} by this overcounting factor, we get the number of desired spaces $W$ to be:
    $$
\prod_{i=j}^{k} \binom{\sum\limits_{i>j}(b_i-a_i) + b_j}{a_j}_q \cdot q^{f(\bar{a},\bar{b})}
$$
where the exponent is
\begin{eqnarray*}
    f(\bar{a},\bar{b}) &=& a_k\left[ n_B - \sum_{i\geq k}(i-k+1)b_i - (k-1)a_k \right] \\
    &+& \sum_{j=1}^{k-1} a_j\left( n_B - \sum_{i\geq j}(i-j+1)b_i -(j-1)a_k - (j-1)a_j
    - 2(j-1)\sum_{i=j}^{k-1} a_i - \sum_{i< j}a_i \right) \\
    &-& a_k\sum_{j=1}^{k-1} ja_j \\
    &=& \sum_{j=1}^{k} a_j\left( n_B - \sum_{i\geq j}(i-j+1)b_i -(j-1)a_j - 2(j-1)\sum_{i\geq j}a_i - \sum_{i<j}a_i\right)
\end{eqnarray*}
completing the proof of the induction step.
\end{proof}

The evaluation formula has a surprising consequence: the size of the largest Jordan block does not in fact matter -- as long as its multiplicity is the same in $A$ and $B$.
\begin{corollary}\label{corollary: largest block}
Suppose that $A$ and $B$ are unipotent Jordan matrices such that the largest Jordan block of $B$ is of size $n$. If $a_n=c$ and $\sum\limits_{i\geq n} b_i=c$, then $X_A(B)$ is independent of $n$. 
\end{corollary}

\section{Product of $q$-character polynomials}\label{sec: recursive formula}
With the evaluations of our statistics explicitly calculated, we move to describe a recursive procedure for determining the product expansions of such $q$-character polynomials. The main consequences of the recursive formula are that many expansion coefficients must vanish, and that expansion coefficients are often polynomial in $q$.

We are seeking to determine the coefficients $r_{A,B}^C$ in the expansion 
$$
X_{A}\cdot X_{B} = \sum_{C}r_{A,B}^C X_C.$$

\subsection{Recursive formula for product expansions}
Our approach proceeds via the following key observation.
\begin{proposition}\label{prop:A smaller than B}
Given two matrices $A$, $C$ where $\dim(C)\leq \dim(A)$, then $X_A(C)=0$ unless $A\sim C$, in which case $X_A(C)=1$.
\end{proposition}
\begin{proof}
If we consider any subspace $W\subseteq \F^{\dim(C)}$, then $\dim(W)\leq \dim(c)\leq \dim(A)$. By the definition of $X_A$, if a subspace $W$ is counted by $X_A(C)$ then $\dim(W)=\dim(A)$ and therefore $\dim(A)=\dim(C)$. Otherwise, when $\dim(C)<\dim(A)$ no spaces are counted and $X_A(C)=0$. If $X_A(C)\neq 0$, then $W$ must be the entire space $\F^{\dim(C)}$, so the condition $C|_W\sim A$ simplifies to $C\sim A$. In this case, only one subspace is counted so $X_A(C)=1$.
\end{proof}

The proposition gives a recursive formula for the coefficients $r^C_{A,B}$.

\begin{theorem}[\textbf{Recursive formula for expansion coefficients}]\label{theorem: recursion formula}
For any matrix $C$, the expansion coefficient 
$r^C_{A,B}$ is given by,
\begin{equation} \label{eq:recursive formula}
r^C_{A,B} = X_{A}(C) \cdot X_{B}(C)-\sum_{\dim(M)< \dim(C)} r_{A,B}^M X_M(C).
\end{equation} 
The claim now follows from the previous proposition.
\end{theorem}
\begin{proof}[Proof of Theorem~\ref{theorem: recursion formula}]
By evaluating the statistics in the equation described in Theorem \ref{theorem-main: product expansions} at a given conjugacy class of a matrix $C$, we see that 
\[
X_{A}(C) \cdot X_{B}(C)=r^C_{A,B}X_{C}(C)+\sum_{\dim(M)< \dim(C)} r_{A,B}^M X_M(C).
\]
\end{proof}

This recursive formula allows us to apply the following inductive procedure to calculate the coefficients.

\par \begin{enumerate}
  \item Start with $k=\max\{\dim(A),\dim(B)\}$.
  \item List the conjugacy classes of matrices $C$ with $\dim(C)=k$, and for each class calculate a Jordan form for $C$ (possibly in a field extension).
  \item Evaluate $X_{A}$, $X_{B}$ and $X_{M}$ for all $M$ with $\dim(M)<\dim(C)$ appearing in \eqref{eq:recursive formula} using the formula in Theorem~\ref{theorem: reducing to unipotents}. Then Theorem~\ref{theorem: recursion formula} determines $r^C_{A,B}$.
  \item Increment $k$ by 1 and return to Step (2). Repeat these steps until $k=\dim(A)+\dim(B)$.
\end{enumerate} 

\subsection{Properties of expansion coefficients}
Many of the coefficients in the aforementioned product expansion are, in fact, zero. For example, a necessary condition for the nonvanishing of a coefficient $r^C_{A,B}$ is that the largest Jordan block of $C$ is exactly the same size as the largest one among those of $A$ and $B$. 
To prove this and other vanishing results for the expansion coefficients,
we posit the following general strategy.

\begin{lemma}\label{lemma: meta}
    Suppose $\mathcal P$ is a collection of conjugacy classes in $\bigcup_{n\geq 1}\Gl_n(\F_q)$ such that $A,B \in \mathcal{P}$ and for all $C \notin \mathcal{P}$, there exists a conjugacy class $C'$ of strictly smaller dimension such that $X_{C'}(C)=1$ and $X_D(C)=X_D(C')$ for any $D \in \mathcal P$. Then $r_{A,B}^{(-)}$ is supported on $\mathcal{P}$, i.e. $r_{A,B}^C \neq 0$ only if $C \in \mathcal{P}$.  
\end{lemma}
\begin{proof}
We show that $r_{A,B}^C = 0$ whenever $C \notin \mathcal{P}$ by induction. List all of the conjugacy classes of matrices with dimensions between $1$ and $n_1+n_2$ inclusive, $C_1, C_2, \ldots, C_k$ such that $\dim C_1 \leq \dim C_2 \leq \cdots \leq \dim C_k$. 

Assume by induction that for all $i\leq m$, the coefficient $r_{A,B}^{C_i}$ of $X_{C_i}$ is $0$ whenever $C_i \notin \mathcal{P}$, we will prove the same holds for $C_{m+1}$. From Theorem~\ref{theorem: recursion formula},
\[ r_{A,B}^{C_{m+1}} =  X_{A}(C_{m+1}) \cdot X_{B}(C_{m+1}) - \sum_{i=1}^m r_{A,B}^{C_i} X_{C_i} (C_{m+1}) .\]

If $C_{m+1} \notin \mathcal{P}$, then by assumption there exists a smaller matrix $C_{\ell}$ with $\ell < m+1$ such that $X_A(C_{m+1}) = X_A(C_{\ell})$ and $X_B(C_{m+1}) = X_B(C_{\ell})$. Additionally, for all $1 \leq i \leq m$ such that $r_{A,B}^{C_i}$ is nonzero, the induction hypothesis implies that $C_i \in \mathcal{P}$, and therefore $X_{C_i}(C_{m+1}) = X_{C_i}(C_\ell)$. Furthermore, for all $i$ with $\ell+1 \leq i \leq m$, Proposition~\ref{prop:A smaller than B} implies that $X_{C_i}(C_\ell) = 0$ since $\dim(C_\ell) \leq \dim(C_i)$ and $C_i \neq C_\ell$. Therefore,
\begin{align*}
    r_{A,B}^{C_{m+1}} &= X_A(C_{m+1}) X_B(C_{m+1}) - \sum_{i=1}^m r_{A,B}^{C_i} X_{C_i}(C_{m+1}) \\
    &= X_A(C_{m+1}) X_B(C_{m+1}) - \sum_{i=1}^{\ell-1} r_{A,B}^{C_i} X_{C_i}(C_{m+1}) - r_{A,B}^{C_\ell} X_{C_\ell}(C_{m+1}) -\sum_{i=\ell+1}^m r_{A,B}^{C_i} X_{C_i}(C_{m+1}) \\
    &= X_A(C_\ell) X_B(C_\ell) - \sum_{i=1}^{\ell-1} r_{A,B}^{C_i} X_{C_i} (C_\ell)  - r_{A,B}^{C_\ell} = 0,
\end{align*}
where the last equality follows from Theorem~\ref{theorem: recursion formula} applied to $r_{A,B}^{C_\ell}$. This completes the induction step and we conclude that the coefficient $r_{A,B}^C$ is only nonzero if $C \in \mathcal{P}$.
\end{proof}

Let us apply Lemma \ref{lemma: meta} to a few notable examples for the collection $\mathcal{P}$.
\begin{corollary}\label{corollary: no new eigenvalues}
If $C$ contains any eigenvalue that is not an eigenvalue of $A$ or $B$, then $r_{A,B}^C=0$.
\end{corollary}
\begin{proof}
Apply Lemma~\ref{lemma: meta} to the collection of conjugacy classes $\mathcal P$ that only contain eigenvalues that appear in either $A$ or $B$. For every $C \notin \mathcal{P}$, let $C'$ be the restriction of $C$ to the maximal subspace on which $C$ acts with eigenvalues that appear in either $A$ or $B$. 
Then, by Lemma~\ref{lem:evaluations of disjoint eigenvalues}, $X_{D}(C)=X_{D}(C')$ and $X_{C'}(C) = X_{C'}(C')=1$ for all $D \in \mathcal{P}$. The conditions of Lemma ~\ref{lemma: meta} are met, and we are done.
\end{proof}

\begin{corollary}\label{corollary: largest Jordan block}
Let $A$ and $B$ be unipotent Jordan matrices such that the largest Jordan blocks of $A$ and $B$ are $a$ and $b$, respectively, with $a \leq b$. Consider the expansion
\[ X_A \cdot X_B = \sum_C r_{A,B}^C X_C.\]
If $r_{A,B}^C \neq 0$, then the largest Jordan block of $C$ has size exactly $b$.
\end{corollary}
This restriction
greatly reduces the necessary calculations involving the recursive algorithm described above. The simplification will feature prominently in the next section.
\begin{proof}
By Corollary~\ref{corollary: no new eigenvalues}, $r_{A,B}^C$ can only be nonzero when $C$ is unipotent. We apply Lemma~\ref{lemma: meta} to the collection of unipotent matrices $\mathcal{P}$ with largest Jordan block of size $\leq b$. For any $C\notin P$ define $C'$ to be the result of replacing all Jordan blocks of size $>b$ in $C$ to ones of size $b$. By Corollary~\ref{corollary: largest block}, it is clear that $X_D(C')=X_D(C)$ for any $D\in \mathcal{P}$. Furthermore, by Theorem~\ref{theorem:general evaluations} it is easy to verify that $X_C(C')=1$. Therefore, the collection $\mathcal{P}$ satisfies the conditions of Lemma $\ref{lemma: meta}$, so $r_{A,B}^C \neq 0$ only when the largest Jordan block of $C$ has size $\leq b$. 

We next prove that there are no matrices $C$ with largest Jordan block of size $<b$ for which $r^C_{A,B}\neq 0$. For the sake of contradiction, assume the contrary and let $C$ be a counterexaple of minimal dimension, i.e. $r_{A,B}^C\neq 0$ and the maximal block size in $C$ is $<b$,  while for any matrix of strictly smaller dimension with maximal block size $<b$ the expansion coefficient vanishes.
Substitute $C$ into the product expansion. On $X_A\cdot X_B$ the evaluation is $0$ since it is so on $X_B$. On the other hand, the evaluation must be $r^C_{A,B}$ as $X_C(C)=1$ and all other terms evaluate to 0 by minimality of $C$. We arrive at a contradiction as $r^C_{A,B}=0$. Altogether, if $r_{A,B}^C \neq 0$ then largest Jordan block of $C$ has size $b$.
\end{proof}

\begin{corollary}[\textbf{Polynomiality of expansion coefficients}]\label{corollary: polynomial expansion coeff}
Consider square matrices $A$ and $B$ 
such that all their eigenvalues are contained in $\mathbb{F}_{q}$. Then considered as matrices over $\F_{q^m}$, the expansion coefficient $r_{A,B}^C\neq 0$ only for conjugacy classes of matrices $C$ over $F_q$, and for every such $C$ there exists an integer polynomial $R_{A,B}^C(x)\in \mathbb{Z}[x]$ such that for every field extension $\F_{q^m}$ the product expansion coefficient is the evaluation, $r_{A,B}^C = R_{A,B}^C(q^m)$.

In other words, the product expansion has coefficients polynomial in $q$,
\[
X_A \cdot X_B = \sum_{C \text{ over }\F_q} R_{A,B}^C(q^{m}) X_C.
\]
where the statistics are considered as matrices over $\F_{q^m}$.
\end{corollary}

\begin{proof}

Suppose without loss of generality that $A$ and $B$ have dimensions $n_A \leq n_B$, respectively. By Corollary~\ref{corollary: no new eigenvalues}, it follows that the coefficient $r_{A,B}^C$ in the product expansion is non-zero only if $C$ shares its eigenvalues with $A$ and $B$, thus only when $C$ has eigenvalues in $\F_q$. As a result, $r_{A,B}^C \neq 0$ only when $C$ is conjugate to a matrix over $\F_q$. So, it suffices to consider conjugacy classes of $\F_q$-matrices.

List all of the conjugacy classes of $\F_q$-matrices with dimensions between $n_B$ and $n_A+n_B$ as $C_1, C_2, \ldots, C_r$ ordered such that $\dim C_1 \leq \dim C_2 \leq \cdots \leq \dim C_r$. We show by induction that the coefficient of $X_{C_i}$ is a polynomial in $q^{m}$. For the base case $i=1$, Theorem~\ref{theorem: recursion formula} shows that the coefficient $r_{A,B}^{C_1}$ is $X_{A}(C_1, q^m) \cdot X_{B} (C_1, q^m)$.
By Theorem~\ref{theorem: reducing to unipotents}, both $X_{A}(C_1, q^m)$ and $X_{B} (C_1, q^m)$ are integer polynomials evaluated at $q^m$, and therefore so is $r^{C_1}_{A,B}$. Denote this polynomial by $R^{C_1}_{A,B}(t)$, so that $r^{C_1}_{A,B}=R^{C_1}_{A,B}(q^m)$.


Assume by induction that for $i \leq k$, the coefficient $r^{C_i}_{A,B}$ coincide with integer polynomials $R^{C_i}_{A,B}(q^{m})$. Then, by Theorem \ref{theorem: recursion formula} the coefficient
\begin{align*} r_{A,B}^{C_{k+1}} &= X_{A}(C_{k+1}, q^m) \cdot X_{B}(C_{k+1}, q^m) - \sum_{i=1}^k R_{A,B}^{C_i}(q^{m}) X_{C_i} (C_{k+1}, q^m).
\end{align*}

By Corollary~\ref{corollary: no new eigenvalues}, the eigenvalues in $C_i$ must be eigenvalues of $A$ or $B$, so by Theorem~\ref{theorem: reducing to unipotents}, all three of $X_{A}(C_{k+1}, q^m)$, $X_{B}(C_{k+1}, q^m)$, and $X_{C_i} (C_{k+1}, q^m)$ are integer polynomials in $q^m$. It follows that $r^{C_{k+1}}_{A,B}$ is an integer polynomial in $q^m$ as well, thus completing the induction and the proof.
\end{proof}

\begin{proof}[Proof of Theorem~\ref{theorem-main: product expansions} and Theorem~\ref{theorem-main:recursive description}]
By Corollary~\ref{corollary: polynomial expansion coeff}, the product expansions coefficients of $X_{J_{\lambda}}$ and $X_{J_{\mu}}$  for unipotent matrices $J_{\lambda}$ and $J_{\mu}$ coincide with integer polynomials $R_{J_\lambda,J_\mu}^{C}(x)\in \mathbb{Z}[x]$ evaluated at $q^m$. Furthermore, Corollary \ref{corollary: no new eigenvalues} shows that all matrices $C$ that contribute nontrivially to the product expansion are again unipotent. Simplify the notation as $R_{\lambda\mu}^{\nu}(x) := R_{J_\lambda,J_{\mu}}^{J_{\nu}}(x)$.

Recalling that $P_{\lambda\nu}(x)\in \mathbb{Z}[x]$ are taken so that $P_{\lambda\nu}(p^d) = X_{J_{\lambda}}(J_{\nu},p^d)$ for all prime powers, the recursive formula of Theorem \ref{theorem: recursion formula} shows
\[ R_{\lambda\mu}^{\nu}(p^d) = P_{\lambda\nu}(p^d) \cdot P_{\mu\nu}(p^d) - \sum_{\nu'<\nu} R_{\lambda\mu}^{\nu'}(p^d) P_{\nu'\nu}(p^d).
\]
Since this equality holds for infinitely many evaluations, it already holds in the ring $\mathbb{Z}[x]$. From this recursion in $\mathbb{Z}[x]$ it is clear that the polynomials $R^\nu_{\lambda\mu}(x)$ are expressible in terms of $P_{\nu'\nu''}(x)$ for various $\nu'$ and $\nu''$. In particular $R^{\nu}_{\lambda\mu}(x)$ does not depend on the characteristic $p$, as none of $P_{\nu'\nu''}(x)$ do.
\end{proof}

The sense in which the polynomials $R_{\lambda,\mu}^\nu(x)$ determine the product structure of $q$-character polynomials is given by the following formula.
\begin{theorem}\label{theorem: general coefficients in terms of unipotents}
Consider square matrices $A$ and $B$ such that their eigenvalues are contained in the field $\F_q$. Let $\alpha_1,\ldots,\alpha_k$ be the set of distinct eigenvalues of $A$ and $B$, and let $\lambda_i$ and $\mu_i$ denote the partitions enumerating the Jordan block sizes with eigenvalue $\alpha_i$ in $A$ and $B$, respectively. Then,
\begin{equation}
     X_A \cdot X_B = \sum_{\nu_1,\ldots,\nu_k} 
\left(\prod_{i=1}^k R_{\lambda_i,\mu_i}^{\nu_i}(q) \right)X_{J_{\nu_1}(\alpha_1)\oplus\cdots\oplus J_{\nu_k}(\alpha_k)}
\end{equation}
where the sum goes over all $k$-tuples of partitions $\nu_1,\ldots,\nu_k$ with $\sum \|\nu_i\| \leq \dim(A)+\dim(B)$.
\end{theorem}
\begin{proof}
By Lemma~\ref{lem:product with disjoint eigenvalues},
\[ X_A \cdot X_B = \prod_{i=1}^k X_{J_{\lambda_i}(\alpha_i)} \cdot \prod_{i=1}^k X_{J_{\mu}(\alpha_i)}= \prod_{i=1}^k X_{J_{\lambda_i}(\alpha_i)} X_{J_{\mu_i}(\alpha_i)} .\]
Let us now show that
\[ X_{J_{\lambda_i}(\alpha_i)} X_{J_{\mu_i}(\alpha_i)} =  \sum_{\nu_i} R_{\lambda_i, \mu_i}^{\nu_i}(q) X_{J_{\nu_i}(\alpha_i)} .\]
Proceed by induction on $\|\nu_i\|$ with the recursive formula \eqref{eq:recursive formula}. Suppose that for all $\nu_i'$ with $\|\nu_i'\|<\|\nu_i\|$, we have $R_{J_{\lambda_i}(\alpha_i),J_{\mu_i}(\alpha_i)}^{J_{\nu_i'}(\alpha_i)}(q) = R_{\lambda_i, \mu_i}^{\nu_i'}(q)$. Then, by \eqref{eq:recursive formula},
\begin{align*}
    R_{J_{\lambda_i}(\alpha_i),J_{\mu_i}(\alpha_i)}^{J_{\nu_i'}(\alpha_i)}(q) &= X_{J_{\lambda_i}(\alpha_i)}(J_{\nu_i}(\alpha_i)) \cdot X_{J_{\mu_i}(\alpha_i)}(J_{\nu_i}(\alpha_i))-\sum_{\|\nu_i'\|< \|\nu_i\|} R_{J_{\lambda_i}(\alpha_i),J_{\mu_i}(\alpha_i)}^{J_{\nu_i'}(\alpha_i)}(q)X_{J_{\nu_i'}(\alpha_i)} (J_{\nu_i}(\alpha_i)) \\
    &= X_{J_{\lambda_i}}(J_{\nu_i}) \cdot X_{J_{\mu_i}}(J_{\nu_i})-\sum_{\|\nu_i'\|< \|\nu_i\|} R_{\lambda_i,\mu_i}^{\nu_i'}(q) X_{J_{\nu_i'}} (J_{\nu_i}) \\
    &= R_{\lambda,\mu}^{\nu}(q),
\end{align*}
where the second equality follows from Lemma~\ref{lem: changing eigenvalue}. This completes the induction. We conclude that 
\begin{align*}
     X_A \cdot X_B &= \prod_{i=1}^k X_{J_{\lambda_i}(\alpha_i)} X_{J_{\mu_i}(\alpha_i)} \\
     &= \prod_{i=1}^k \sum_{\nu_i} R_{\lambda_i, \mu_i}^{\nu_i}(q) X_{J_{\nu_i}(\alpha_i)} \\
     &= \sum_{\nu_1,\ldots,\nu_k} \left(\prod_{i=1}^k R_{\lambda_i,\mu_i}^{\nu_i}(q) \right)X_{J_{\nu_1}(\alpha_1)\oplus\cdots\oplus J_{\nu_k}(\alpha_k)},
\end{align*}
where we have rewritten the product as a sum in the last step using Lemma~\ref{lem:product with disjoint eigenvalues}.
\end{proof}

\section{Calculations of expansion coefficients}\label{sec: expansion coefficients}
We now use our recursive description of the product expansion coefficients to study their general properties and compute specific product expansion coefficients.

\subsection{Small cases}
We first endeavor to find the explicit product expansion coefficients for the product of statistics on unipotent matrices with small numbers of Jordan blocks.
 
\begin{example}[Single Jordan blocks.] \label{example:Ja*Jb}
For $b \geq a$, the following expansion holds.

    \begin{equation}\label{eq:Ja*Jb}
    X_{J_a} \cdot X_{J_b} = \begin{cases} X_{J_b}+ \left(\sum_{m=1}^{a-1} q^{2m-1}(q-1) X_{J_{m,b}}\right) + q^{2a} X_{J_{a,b}} & b>a \\
    X_{J_b}+\left(\sum_{m=1}^{a-1} q^{2m-1}(q-1) X_{J_{m,b}}\right) + q^{2a-1}(q+1) X_{J_{a,b}} & b= a
    \end{cases}
    \end{equation}
\end{example}
\begin{proof}
We will only prove the case when $b>a$ as the other case is similar.

One can find the coefficients in Equation~\eqref{eq:Ja*Jb} recursively using the algorithm in Section~\ref{sec: recursive formula}. However, it suffices to verify that Equation \eqref{eq:Ja*Jb} satisfies the recursion in Theorem \ref{theorem: recursion formula} for all the matrices $C$ such that $R_{J_a, J_b}^C(q)$ may be nonzero. By Corollary~\ref{corollary: largest Jordan block} and the fact that the dimension of $C$ is at most $a+b$, any such $C$ must be of the form $J_{1^{c_1},2^{c_2},\ldots, a^{c_a}, b}$ with $\sum_{j=1}^{a} jc_j \leq a$.


Consider any such $C=J_{1^{c_1},2^{c_2},\ldots, a^{c_a}, b}$, letting $k = \sum_{j=1}^{a}j c_j$ denote the size of the matrix $C$ without its largest block. Note that by Lemma~\ref{lemma:equal blocks}, $X_{J_b}(C) = q^k$ and $X_{J_a}(C) = \binom{c_a+1}{1} q^{k-c_a}$.
Now, consider the evaluations $X_{J_{m,b}} (C)$ with $1 \le m \le a$. Theorem~\ref{theorem:general evaluations} implies,%
\begin{align*}X_{J_{m,b}}(C) &= \binom{\sum\limits_{i\geq m}c_i}{1}_q  q^{k+b - \left((b-m+1)+\sum\limits_{i\geq m}(i-m+1)c_i\right) -(m-1) -2(m-1)} \cdot q^{k-1}\\
&= \binom{\sum\limits_{i\geq m}c_i}{1}_q \cdot q^{2k+1-2m +\sum\limits_{i\geq m}(m-1-i)c_i}.
\end{align*}
Using the identity $(q-1)\binom{c}{1}_q = q^c-1$ we conclude,
\begin{align*}
q^{2m-1}(q-1)X_{J_{m,b}}(C) &= q^{2k}\cdot \left( q^{\sum_{i\geq m}c_i} - 1 \right)q^{\sum_{i\geq m} (m-1-i)c_i} \\
&= q^{2k}\left( q^{\sum_{i\geq m+1} (m-i)c_i} - q^{\sum_{i\geq m} (m-1-i)c_i} \right).
\end{align*}
The sum of these expressions for $1 \leq m \leq a$ telescopes, finally yielding,
\begin{equation} \label{eq:telescope}
\sum_{m=1}^{a} q^{2m-1}(q-1)X_{m,b}(C) = q^{2k}\left( q^{0} - q^{\sum_{i\geq 1}-ic_i}  \right) = q^{2k}\left(1-q^{-k} \right) = q^{k}(q^k-1).
\end{equation}
Lastly, adding $X_b(C)=q^k$ and $q^{2a-1}X_{a,b}(C)=\binom{c_a}{1}_q \cdot q^{2k-c_a }$ to \eqref{eq:telescope} gives the LHS of \eqref{eq:Ja*Jb} to be
\[
= q^{2k} + \binom{c_a}{1}_q\cdot q^{2k-c_a} = q^{2k-c_a}\left( q^{c_a}\binom{c_a}{0}_q + \binom{c_a}{1}_q \right) = q^{2k-c_a}\binom{c_a+1}{1}_q = X_{J_a}(C) \cdot X_{J_b}(C),
\]
where the third equality follows from Pascal's identity for $q$-binomial coefficients. From this equality we conclude that there is no remaining non-trivial linear contributions of any $X_C$, and thus Equation~\eqref{eq:Ja*Jb} is indeed satisfied by all conjugacy classes.
\end{proof}




\begin{remark}
Observe that in the expansion $X_{J_a}\cdot X_{J_{b}}$, where $b \ge a$, the only non-trivial contributions to the expansion are from conjugacy classes with one or two Jordan blocks.
Furthermore, the formula implies is that any statistic associated with a matrix with two Jordan blocks can be generated by statistics on single Jordan blocks.

\begin{corollary}
All statistics $X_{J_{a,b}}$ can be expressed as degree 2 polynomials in $\{X_{J_{n}} | n \in \mathbb{N}\}$ with coefficients in rational functions $\in \mathbb{Q}(q)$.
\end{corollary}
See \S\ref{subsec: conjectures - number of blocks} for a conjectural generalization.
\end{remark}


\begin{example}[Jordan matrices with $2$ or fewer blocks] \label{lemma:Jb*Ja1}
Using similar methods as the previous example, one can determine the expansion coefficients for the product of a statistic on a single Jordan block with a statistic on two Jordan blocks. For example, given $b>a>1$,
\begin{align}\label{eq:larger jordan expansion}
\begin{split}
    X_{J_{b}} \cdot &X_{J_{1,a}} = q X_{J_{1,b}} + \sum_{m=2}^{a-1} q^{2m-2}(q-1) X_{J_{m,b}} + q^{2a-1} X_{J_{a,b}} \\
    &+ q^2(q^2-1)X_{J_{1^2,b}} + \sum_{m=2}^{a-1} q^{2m+1}(q-1)X_{J_{1,m,b}} + q^{2a+2} X_{J_{1,a,b}}.
\end{split}
\end{align}
\end{example}
\noindent \textit{Proof sketch.} A similar proof to that of Example \ref{example:Ja*Jb} shows that the left and right hand sides are equal as desired.
\begin{example}[More matrices]
The following table collects calculations concerning small cases that fall outside the above general cases. This table is generated with the premise that $b$ is greater than the size of the maximal Jordan block in Term 2, a condition necessary for the expansions to stabilize. 

\begin{table}[h]\centering \label{table:product expansions}
{\renewcommand{\arraystretch}{1.4}
\fontsize{9.5}{10} \selectfont
\begin{tabular}{|c|c|c|}
\hline
Term 1 & Term 2     & Product Expansion                                                                                                                                                            \\ \hline
$X_{J_b}$   & $X_{J_{2,2}}$   & 

$q^2X_{J_{b, 2}}+(q-1)q^4X_{J_{b, 2, 1}}+q^8X_{J_{b, 2, 2}}
$                \\ \hline
$X_{J_b}$   & $X_{J_{3,2}}$   & \shortstack{$q^2X_{J_{b, 2}}+q^4X_{J_{b, 3}}+(q-1)q^4X_{J_{b, 2, 1}}+(q^2-1)q^6X_{J_{b, 2, 2}}$ \\ $+(q-1)q^6X_{J_{b, 3, 1}}+q^{10}X_{J_{b, 3, 2}}$} \\ \hline
$X_{J_b}$   & $X_{J_{4,2}}$   & \begin{tabular}{@{}c@{}} $q^2X_{J_{b,2}}+(q-1)q^3X_{J_{b,3}}+q^6X_{J_{b,4}}+ q^4(q-1)X_{J_{b,2,1}}$\\$+(q-1)^2q^5X_{J_{b,3,1}}+ (q^2-1)q^6X_{J_{b,2,2}} + q^9(q-1)X_{J_{b,3,2}}$\\$+q^8(q-1)X_{J_{b,4,1}}+q^{12}X_{J_{b,4,2}}$\end{tabular}                                                                                                                                                                                                                 \\ \hline
$X_{J_b}$   & $X_{J_{5,2}}$   &   \begin{tabular}{@{}c@{}}        $q^{2}X_{J_{b, 2}}+10q^4X_{J_{b, 2, 1}}+(q-1)q^3X_{J_{b, 3}}+(q^2-1)q^6X_{J_{b, 2, 2}}$\\ $+(q-1)^2q^5X_{J_{b, 3, 1}}+(q-1)q^5X_{J_{b, 4}}+(q-1)q^9X_{J_{b, 3, 2}}$\\$+(q-1)^2q^7X_{J_{b, 4, 1}}+q^{8}X_{J_{b, 5}}+(q-1)q^{11}X_{J_{b, 4, 2}}$\\$+(q-1)q^{10}X_{J_{b, 5, 1}}+q^{14}X_{J_{b, 5, 2}}$     \end{tabular}                                                                         \\ \hline   
$X_{J_b}$   & $X_{J_{6,2}}$   &  \begin{tabular}{@{}c@{}}$q^{2}X_{J_{b, 2}}+(q-1)q^4X_{J_{b, 2, 1}}+(q-1)q^3X_{J_{b, 3}}+(q^2-1)q^6X_{J_{b, 2, 2}}$\\$+(q-1)^2q^5X_{J_{b, 3, 1}}+(q-1)q^5X_{J_{b, 4}}+(q-1)q^9X_{J_{b, 3, 2}}$\\$+(q-1)^2q^7X_{J_{b, 4, 1}}+(q-1)q^7X_{J_{b, 5}}+(q-1)q^{11}X_{J_{b, 4, 2}}$\\$+(q-1)^2q^9X_{J_{b, 5, 1}}+q^{10}X_{J_{b, 6}}+(q-1)q^{13}X_{J_{b, 5, 2}}$\\$+(q-1)q^{12}X_{J_{b, 6, 1}}+q^{16}X_{J_{b, 6, 2}}$       \end{tabular}                                                                                           \\ \hline
$X_{J_b}$   & $X_{J_{3,3}}$   &         $q^{3}X_{J_{b, 3}}+(q-1)q^5X_{J_{b, 3, 1}}+(q-1)q^8X_{J_{b, 3, 2}}+q^{12}X_{J_{b, 3, 3}}$                                                             \\ \hline
    $X_{J_b}$   & $X_{J_{4,3}}$   &         \begin{tabular}{@{}c@{}}  $q^{3}X_{J_{b, 3}}+q^{5}X_{J_{b, 4}}+(q-1)q^5X_{J_{b, 3, 1}}+(q-1)q^8X_{J_{b, 3, 2}}$\\$+(q-1)q^7X_{J_{b, 4, 1}}+(q^2-1)q^{10}X_{J_{b, 3, 3}}+(q-1)q^{10}X_{J_{b, 4, 2}}+q^{14}X_{J_{b, 4, 3}} $   \end{tabular}                                                                                                                                                                                                                                                              \\ \hline
$X_{J_b}$   & $X_{J_{5,3}}$   &     \begin{tabular}{@{}c@{}}   $q^{3}X_{J_{b, 3}}+(q-1)q^5X_{J_{b, 3, 1}}+(q-1)q^4X_{J_{b, 4}}+(q-1)q^8X_{J_{b, 3, 2}}$\\$+(q-1)^2q^6X_{J_{b, 4, 1}}+q^{7}X_{J_{b, 5}}+(q^2-1)q^{10}X_{J_{b, 3, 3}}+(q-1)^2q^9X_{J_{b, 4, 2}}$\\$+(q-1)q^9X_{J_{b, 5, 1}}+(q-1)q^{13}X_{J_{b, 4, 3}}+(q-1)q^{12}X_{J_{b, 5, 2}}+q^{16}X_{J_{b, 5, 3}}$   \end{tabular}                                                                                                                                                                                                                                                                              \\ \hline
$X_{J_b}$   & $X_{J_{4,4}}$   &                    \begin{tabular}{@{}c@{}}     $ q^{4}X_{J_{b, 4}}+(q-1)q^6X_{J_{b, 4, 1}}+(q-1)q^9X_{J_{b, 4, 2}}$\\$+(q-1)q^{12}X_{J_{b, 4, 3}}+q^{16}X_{J_{b, 4, 4}}$\end{tabular}                                                                                                                                              \\ \hline
$X_{J_b}$   & $X_{J_{2,2,2}}$ &   $q^{4}X_{J_{b, 2, 2}}+(q-1)q^7X_{J_{b, 2, 2, 1}}+q^{12}X_{J_{b, 2, 2, 2}}$                                                                                                                 \\ \hline 
$X_{J_b}$   & $X_{J_{3,3,2}}$ &       \begin{tabular}{@{}c@{}}  $q^{5}X_{J_{b, 3, 2}}+(q-1)q^8X_{J_{b, 3, 2, 1}}+q^{8}X_{J_{b, 3, 3}}$\\$+(q-1)q^{11}X_{J_{b, 3, 3, 1}}+(q^2-1)q^{11}X_{J_{b, 3, 2, 2}}+q^{16}X_{J_{b, 3, 3, 2}}$                                          \end{tabular}                                                                                                                          \\ \hline
\end{tabular}}
\caption{Table of Product Expansions}
\end{table}
\end{example}

\subsection{Observations and conjectures}\label{subsec: conjectures}

The examples computed above exhibit a number of patterns that we did not anticipate at the start of this project. We take this section to record some of our observations and suggest new conjectural structure on the ring of $q$-character polynomials.

\subsection{Stabilization for large blocks}
First, observe that in the examples above, the size of the large Jordan block $J_b$ plays only minor role in the expansion: once $b$ is larger than the dimension of the second matrix $J_{a_1,\ldots,a_k}$, the product expansion takes a fixed form but for the fact that all entries have a single `leading' Jordan block of size $b$.

\begin{proposition}[Independence of largest block]
For every partition $\lambda = (a_1\geq \ldots \geq a_k)$ there is a fixed sequence of partitions $\mu^{(1)},\ldots,\mu^{(r)}$ and polynomials $p_1(t),\ldots,p_r(t)$, such that for all $b> \| \lambda \|$ one has
\[
X_{J_b} X_{J_{\lambda}} = \sum_{i=1}^r p_i(q)X_{J_b \oplus J_{\mu^{(i)}}}.
\]
\end{proposition}
\begin{proof}
Due to Corollary~\ref{corollary: largest Jordan block} and Theorem~\ref{theorem-main: product expansions}, there exists an expansion
\[ X_{J_b} X_{J_{\lambda}} = \sum_{\|\mu \| \leq \|\lambda\| } R_{J_b,J_{\lambda}}^{J_b \oplus J_{\mu}}(q) X_{J_b \oplus J_{\mu}} .\]
We will show with induction on $\|\mu\|$ that the polynomials $R_{J_b,J_{\lambda}}^{J_b \oplus J_{\mu}}(q)$ are indeed independent of $b$ once $b>\|\lambda\|$. For the base case, when $\|\mu\|=0$, we obtain $R_{J_b,J_{\lambda}}^{J_b}(q)=X_{J_{\lambda}}(J_b,q)$, which is independent of $b$ once $b\geq a_1$ as noted in Theorem~\ref{theorem:general evaluations}. Now for partition $\mu$, assume by induction that the polynomial $R_{J_b,J_{\lambda}}^{J_b \oplus J_{\mu'}}(q)$ is independent of $b$ for all $\|\mu'\| < \|\mu\|$. We obtain from Theorem~\ref{theorem: recursion formula},
\begin{equation}\label{eq: independence of largest block recursion}
    R_{J_b,J_{\lambda}}^{J_b \oplus J_{\mu}}(q) = X_{J_b}(J_{b} \oplus J_{\mu},q) \cdot X_{J_{\lambda}}(J_{b} \oplus J_{\mu},q) - \sum_{\|\mu'\|<\|\mu\|}R_{J_b,J_{\lambda}}^{J_b \oplus J_{\mu'}}(q) X_{J_{b}\oplus J_{\mu'}} (J_{b} \oplus J_{\mu},q).
\end{equation} 
The evaluation $X_{J_{\lambda}}(J_{b}\oplus J_{\mu})$ is independent of $b$ as a result of Theorem~\ref{theorem:general evaluations} since $b \geq a_1$. We next show that $X_{J_{b}\oplus J_{\mu'}} (J_{b} \oplus J_{\mu},q)$ is also independent of $b$ for every $\|\mu'\|\leq\|\mu\|$. In particular this proves that $X_{J_b}(J_{b} \oplus J_{\mu},q)$ is independent of $b$.

Assume $J_{b}\oplus J_{\mu} = J_{1^{m_1}2^{m_2}\cdots b^{m_b}}$ (where $m_b=1$) and $J_{b}\oplus J_{\mu'}=J_{1^{n_1}2^{n_2}\cdots b^{n_b}}$ (where $n_b=1$). Applying Theorem~\ref{theorem:general evaluations}, we obtain
\begin{eqnarray*}
X_{J_{b}\oplus J_{\mu'}} (J_{b} \oplus J_{\mu})=\prod_{j=1}^b \binom{\sum\limits_{i>j} (n_i-m_i)+n_j}{m_j}_q q^{m_j\left(\sum\limits_{i} i n_i - \sum\limits_{i\geq j} (i-j+1)n_i - (j-1)m_j-2(j-1)\sum\limits_{i>j} m_i -\sum\limits_{i<j} m_i\right) } \\
= \prod_{j=1}^b \binom{\sum\limits_{i>j} (n_i-m_i)+n_j}{m_j}_q q^{m_j\left(\sum\limits_{i<j} i n_i - \sum\limits_{i\geq j} (-j+1)n_i - (j-1)m_j-2(j-1)\sum\limits_{i>j} m_i -\sum\limits_{i<j} m_i\right) },
\end{eqnarray*}
from which one immediately observes that except for the term corresponding to $j=b$, every term in the product is independent of $b$. For $j=b$, the $q$-binomial is clearly independent of $b$, and the exponent of $q$ is
\begin{align*}
    & \sum_{i=1}^{b-1} in_i - \cancel{(-b+1)n_b} - \cancel{(b-1)m_b} - 2(b-1) \cdot 0  -\sum_{i=1}^{b-1}m_i \\
    =& \sum_{i=1}^{b-1} in_i-m_i = \|\mu'\| - \#(\text{parts of } \mu)
\end{align*}
which is also independent of $b$. Thus, each polynomial in the RHS of equation~\eqref{eq: independence of largest block recursion} is independent of $b$, which implies $R_{J_b,J_{\lambda}}^{J_b \oplus J_{\mu}}(q)$ is independent of $b$ as well.
\end{proof}

We note that this pattern is rather reminiscent of representation stability in that it exhibits an increasing sequence of partitions that differ only by adding one very large leading part. See \cite{church2015fi} for details of that theory.

The pattern further suggest that there is a well-defined completion of $q$-character polynomials, in which one could allow certain blocks of infinite size. It is likely that in this `stable' regime, the product structure of $q$-character polynomials will be more manageable and described in full. A plausible context for this stable range is a Deligne category of $\Gl_t(\F_q)$-representations (for a distinct but suggestive reference, see e.g. \cite{10.1093/imrn/rny144}).

\subsection{Filtration by number of blocks}\label{subsec: conjectures - number of blocks}
Next, we note that in all computed example, the product of statistics $X_{J_\lambda}$ and $X_{J_\mu}$ with $k$ and $\ell$ Jordan blocks, respectively, only has nonzero contributions from statistics with $k+\ell$ or fewer Jordan blocks.
\begin{conjecture}
If $\lambda$ and $\mu$ are any two partitions with $k$ and $\ell$ Jordan blocks, respectively, then the only nonzero expansion coefficients in
$$
X_{J_\lambda}\cdot X_{J_{\mu}} = \sum_{\nu} R_{\lambda,\mu}^\nu(q) X_{J_{\nu}}
$$
are those with $\nu$ having $k+\ell$ or fewer parts.
\end{conjecture}

Conversely, we remarked above that every statistic with $1$ or $2$ Jordan blocks is expressible as a polynomial of degree $2$ in the single-blocked $X_{J_b}$. We conjecture that they in fact generate the entire ring of $q$-character polynomials.
\begin{conjecture}
For every partition $(a_1\geq \ldots \geq a_k)$, there is a polynomial of $P(t_1,\ldots,t_{a_1})$ degree $k$ with coefficients in $\mathbb{Q}(q)$ such that
\[
X_{J_\lambda} = P(X_{J_1},X_{J_2},\ldots,X_{J_{a_1}}).
\]
\end{conjecture}

To further support these conjecture, we extrapolate from the patterns found in Example~\ref{table:product expansions} and propose the following concrete expansions.
\begin{conjecture}
Let $m,n$ be fixed positive integers.
\[X_{J_n} \cdot X_{J_{m,m}}=
\begin{cases}
q^mX_{J_{n,m}}+\sum\limits_{k=1}^{m-1}q^{3k+m-1}(q-1)X_{J_{n,m,k}}+q^{4m}X_{J_{n,m,m}} & \text{if } n>m\\
q^{m-1}\binom{2}{1}_qX_{J_{n,m}}+\sum\limits_{k=1}^{m-1}q^{3k+m-2}(q^2-1)X_{J_{n,m,k}}+q^{4m-2}\binom{3}{1}_qX_{J_{n,m,m}} & \text{if } n=m\\
q^{n-1}\binom{2}{1}_qX_{J_{m,m}}+\sum\limits_{k=1}^{n-1}q^{3k+n-2}(q^2-1)X_{J_{m,m,k}}+q^{4n}X_{J_{m,m,n}} & \text{if } n<m.
\end{cases}\]
\end{conjecture}
\begin{conjecture}
Let $n>m>k$ be positive integers. Then,
\begin{align*}
    &X_{J_n} \cdot X_{J_{m,k}}=q^kX_{J_{n,k}}+\sum\limits_{i=1}^{m-k-1}\big( q^{k+2i-1}(q-1)X_{J_{n,k+i}}\big)+q^{2m-k}X_{J_{n,m}}\\
    &+\sum\limits_{l=1}^{k-1}\left(q^{k+3l-1}(q-1)X_{J_{n,k,l}}+\sum\limits_{j=1}^{m-k-1}\big(q^{k+3k+2j-2}(q-1)^2X_{J_{n,k+j,l}}\big)+q^{2m+3l-k-1}(q-1)X_{J_{n,m,l}}\right)\\
    &+q^{4k-2}(q^2-1)X_{J_{n,k,k}}+\sum\limits_{p=1}^{m-k-1}\big(q^{4k+2p-1}(q-1)X_{J_{n,k+p,k}}\big)+q^{2m+2k}X_{J_{n,m,k}}.
\end{align*}
\end{conjecture}

\section{Application: Correlations of Jordan blocks}\label{subsec: applications}
Our determination of product expansion coefficients allows for calculations of expectations of products of $q$-character polynomials. These, in turn, can be used to calculate joint moments such as the correlation of two of our statistics. 

We first review the simple formula for the expected value of each statistic $X_A$. Suppose $A$ is an invertible $m \times m$ matrix. It is shown in \cite{1803.04155} that the expected value of $X_A$ over all $n \times n$ invertible matrices $B$ is independent of $n$ once $n \geq m$. One may thus calculate this expectation by considering the simplest case when $n=m$. In this case, $X_A(B) = 1$ whenever $A \sim B$ and $X_A(B)=0$ otherwise. It follows that
\begin{equation}\label{eq:expectation}
     \mathbb{E}[X_A] = \frac{|\operatorname{Conj}(A)|}{|\Gl_n(\F_q)|} =  \frac{1}{|C(A)|}
\end{equation}
where $\operatorname{Conj}(A)$ is the conjugacy class of $A$ and $C(A)$ is its centralizer.



It remains to determine the stabilizer of $A$.
\begin{fact}
For matrices in Jordan form with $\leq 2$ blocks, the order of their centralizes are given by,
\begin{eqnarray}
|C(J_b)| &=& q^{b-1}(q-1) \\
|C(J_{b,a})| &=& q^{3a+b-2}(q-1)^2 \quad\text{ when }b>a \\
|C(J_{b,b})| &=& q^{4b-4} (q^2-1)(q^2-q).
\end{eqnarray}
\end{fact}



Recall that the correlation of two random variables is the measure of their association. 
\begin{definition}
The \textbf{correlation} of random variables $X$ and $Y$ is
\[ \rho_{XY} = \frac{\mathbb{E}[XY] - \mathbb{E}[X] \mathbb{E}[Y] }{\sqrt{(\mathbb{E}[X^2]-\mathbb{E}[X]^2)(\mathbb{E}[Y^2]-\mathbb{E}[Y]^2) }} . \]
\end{definition}

In the case of $q$-character polynomials, one can calculate expectations of products such as $\mathbb{E}[XY]$ or $\mathbb{E}[X^2]$ by first applying our product expansions to write the product as a linear combination of $q$-character polynomials. Then, the expected value is determined by \eqref{eq:expectation}. This approach lets us describe fully the correlations between statistics counting single Jordan blocks.


\begin{proof}[Proof of Theorem~\ref{theorem: correlation}]
By definition
\[ \rho = \frac{\mathbb{E}[X_{J_{a}}X_{J_b}] - \mathbb{E}[X_{J_a}] \mathbb{E}[X_{J_b}] }{\sqrt{(\mathbb{E}[X_{J_a}^2]-\mathbb{E}[X_{J_a}]^2)(\mathbb{E}[X_{J_b}^2]-\mathbb{E}[X_{J_b}]^2) }}. \]

First, consider $b>a$. We determine $\mathbb{E}[X_{J_{a}}X_{J_b}]$ and $\mathbb{E}[X_{J_a}^2]$. The former is in Example~\ref{example:Ja*Jb},
\begin{align*}
     \mathbb{E}[X_{J_{a}}X_{J_b}] 
     &= \mathbb{E}[X_{J_b}] + \sum_{i=1}^{a-1} q^{2i-1}(q-1) \mathbb{E}[X_{J_{b,i}}] + q^{2a} \mathbb{E}[X_{J_{b,a}}] \\
     &= \frac{q^a-1}{q^{a+b-2}(q-1)^2} + \mathbb{E}[X_{J_a}] \mathbb{E}[X_{J_b}].
\end{align*}
We find $\mathbb{E}[X_{J_a}^2]$ similarly,
\begin{align*}
    \mathbb{E}[X_{J_{a}}^2] 
     &= \frac{q^a-1}{q^{2a-2}(q-1)^2} + \mathbb{E}[X_{J_a}]^2.
\end{align*}
and similarly for $b$.
Putting these together,
\[ \rho = \frac{ \frac{q^a-1}{q^{a+b-2}(q-1)^2} }{ \sqrt{\frac{q^a-1}{q^{2a-2}(q-1)^2} \cdot \frac{q^b-1}{q^{2b-2}(q-1)^2} } } = \sqrt{\frac{q^a-1}{q^b-1}}.\]

\end{proof}

\bibliographystyle{plain}
\bibliography{references}


\end{document}